\documentclass[11pt,a4paper,reqno]{amsart}
\usepackage[margin=1in]{geometry} 
\usepackage[textsize=tiny]{todonotes} 
\usepackage{float}
\usepackage{amsfonts,amsmath,graphics,epsfig,latexsym,times}
\usepackage{ae,aecompl}
\newif\ifpix \pixtrue
\usepackage{setspace,fancyhdr,amssymb,amsthm,graphicx,ifthen,float,calrsfs}
\usepackage[colorlinks]{hyperref}
\numberwithin{equation}{section} 
\newtheorem{thm}{Theorem}[section]  
\newtheorem{cor}[thm]{Corollary}    
\newtheorem{lem}[thm]{Lemma}        
\newtheorem{prop}[thm]{Proposition}  
\theoremstyle{definition} \newtheorem{dfn}[thm]{Definition}

\newtheorem{assume}[thm]{Assumption}
\newtheorem{prob}[thm]{Problem}
\newtheorem*{claim*}{Claim} 
\newtheorem*{thm*}{Theorem}
\newtheorem{claim}[thm]{Claim} 
\newtheorem{rmk}[thm]{Remark}

\allowdisplaybreaks[1]

\newcommand{\cG}{\mathcal{G}}

\newcommand{\wh}{\widehat}
\newcommand{\wt}{\widetilde}

\newcommand{\cD}{\mathcal{D}}
\newcommand{\cT}{\mathcal{T}}
\newcommand{\cQ}{\mathcal{Q}}
\newcommand{\cE}{\mathcal{E}}

\newcommand{\rd}{{\mathrm d}}

\newcommand{\p}{\partial}
\newcommand{\dbar}{\overline{\partial}}

\DeclareMathOperator{\im}{im}
\DeclareMathOperator{\Ker}{ker}

\DeclareMathOperator{\SO}{SO} 

\newcommand{\CC}{\mathbb{C}}
 
\newcommand{\RR}{\mathbb{R}} 
\newcommand{\ZZ}{\mathbb{Z}}

\newcommand{\bN}{\mathbb{N}}
\newcommand{\bCP}{\mathbb{CP}} 
\newcommand{\bP}{\mathbb{P}}

\newcommand{\cH}{\mathcal{H}}

\newcommand{\cL}{\mathcal{L}}
\newcommand{\cF}{\mathcal{F}}
\newcommand{\cU}{\mathcal{U}}

\newcommand{\cB}{\mathcal{B}}
\newcommand{\cM}{\mathcal{M}}
\newcommand{\cV}{\mathcal{V}}
\newcommand{\cZ}{\mathcal{Z}}
\newcommand{\cW}{\mathcal{W}}

\newcommand{\cP}{\mathcal{P}}

\renewcommand{\phi}{\varphi}

\newcommand{\ve}{\varepsilon}

\newcommand{\defeq}{\mathrel{\mathop:}=}
\renewcommand{\geq}{\geqslant}
\renewcommand{\leq}{\leqslant}

\DeclareMathOperator\SU{SU}

\DeclareMathOperator\id{id}

\author{Joel Fine}

\address{Joel Fine: D\'epartment de math\'ematique, Universit\'e libre de Bruxelles (ULB), CP 218, Boulevard du Triomphe, B-1050 Bruxelles, Belgium}
\email{joel.fine@ulb.ac.be}

\author{Jason D. Lotay}

\address{Jason D. Lotay: Department of Mathematics, University College, London WC1E 6BT, UK}
\email{j.lotay@ucl.ac.uk}

\author{Michael Singer}

\address{Michael Singer: Department of Mathematics, University College, London WC1E 6BT, UK}
\email{michael.singer@ucl.ac.uk}

\date{\today}
\title{The space of hyperk\"ahler metrics on a $4$-manifold with boundary}
\begin{document}

\begin{abstract} 
Let $X$ be a compact 4-manifold with boundary. We study the space of hyperk\"ahler triples $\omega_1,\omega_2,\omega_3$ on $X$, modulo diffeomorphisms which are the identity on the boundary. We prove that this moduli space is a smooth infinite-dimensional manifold and describe the tangent space in terms of triples of closed anti-self-dual 2-forms. We also explore the corresponding boundary value problem: a hyperk\"ahler triple restricts to a closed framing of the bundle of 2-forms on the boundary; we identify the infinitesimal deformations of this closed framing that can be filled in to hyperk\"ahler deformations of the original triple. Finally we study explicit examples coming from gravitational instantons with isometric actions of $\SU(2)$.
\end{abstract} 
\maketitle

\setcounter{tocdepth}{1}
\tableofcontents

\section{Introduction}

In this paper we study hyperk\"ahler metrics on a compact $4$-manifold
$X$ with boundary $\p X = Y$.  Recall that a Riemannian
manifold $(M,g)$, of dimension $4n$, is {\em hyperk\"ahler} if there
exists a triple $(J_1,J_2, J_3)$ of orthogonal complex structures
which satisfy the quaternionic relations
\begin{equation*}
J_1^2 = J_2^2 = J_3^2 = J_1J_2J_3 = -1
\end{equation*}
and such that the corresponding triple $\omega=(\omega_1,\omega_2,\omega_3)$
of $2$-forms is closed, where $\omega_i(\cdot,\cdot) =
g(J_i\,\cdot,\cdot)$.  Thus $(M,g)$ is K\"ahler with respect to each of
the complex structures $J_i$.

When $n=1$, i.e.\ in 4 dimensions, the $\omega_i$ form a flat basis of the bundle of
self-dual $2$-forms $\Lambda^2_+$ and the quaternionic relations
imply
\begin{equation}
\label{omegas_orthogonal}
\omega_i\wedge \omega_j = 2\delta_{ij}\,\mu
\end{equation}
where the volume element $\mu$ is also determined by the triple:
\begin{equation}
\label{hk_volume_form}
\mu = \frac{1}{6}(\omega_1^2 + \omega_2^2 + \omega_3^2).
\end{equation}
Conversely, given a triple of {\em symplectic forms} $\omega =
(\omega_1,\omega_2,\omega_3)$ satisfying \eqref{omegas_orthogonal} and
\eqref{hk_volume_form} we may use the span of the $\omega_i$ to {\em define} a rank 3 sub-bundle of
$\Lambda^2T^*$ on which the natural quadratic form is
positive-definite. By the correspondence between such sub-bundles and
conformal structures in $4$ dimensions, we see that our triple
$\omega$  defines first a conformal structure, then a metric by
declaring $\mu$ in \eqref{hk_volume_form} to be the metric volume form. The conditions \eqref{omegas_orthogonal} and \eqref{hk_volume_form} together with $\rd \omega =0$ then imply that this metric is hyperk\"ahler (cf.\ \cite{Donaldson,Hitchin} and Lemma~\ref{l1.12.8.15} below).

We propose to study hyperk\"ahler metrics in terms of the
corresponding triples.  Our work is inspired in part by
the following local thickening result of
Bryant~\cite{Bryant}.  Suppose that $\iota: Y\subset X$ is a real
hypersurface in the hyperk\"ahler manifold $(X,\omega)$. The pull-back $\gamma
= \iota^*\omega$ is then a {\em closed framing}\footnote{For any closed orientable $3$-manifold $Y$, framings of $\Lambda^2_Y$ exist and a result of Gromov \cite[p.~182]{Gromov} shows that any framing can be deformed to become a closed framing.} of the bundle
$\Lambda^2_Y$ of 2-forms on $Y$.  Conversely, Bryant showed that if $\gamma$ is a
real-analytic closed framing of $\Lambda^2_Y$, then there is a
hyperk\"ahler $4$-manifold $(X,\omega)$ and an embedding $\iota: Y\to
X$ such that $\gamma = \iota^*\omega$. This is a consequence of the Cauchy--Kowalevski theorem applied to a certain evolution equation \eqref{flow.eq} and the reader is referred
to Theorem~\ref{bryant.thm} below for a stronger and more precise statement.  It
has been pointed out to us by Claude LeBrun that the same result can be proved by
twistor-theoretic considerations \cite[Theorem 3.6.I]{Claude_thesis} using 
embeddability results for real-analytic CR manifolds.

Suppose now that $(X,\omega)$ is a compact hyperk\"ahler manifold with boundary
$\partial X = Y$ and inclusion map $\iota:Y\subset X$.  As before, $\gamma =\iota^*\omega$ is a
closed framing of $\Lambda^2_Y$ and our boundary value problem is to
determine which small deformations of $\gamma$ as a closed framing of
$\Lambda^2_Y$ arise as boundary values of hyperk\"ahler deformations
of $\omega$.

There is a well-known analogue of this problem in the context of
self-dual and anti-self-dual (ASD) Einstein metrics of negative scalar
curvature. Here one is given a conformally compact ASD Einstein metric $g$
on the interior of $X$ with conformal infinity $h$, and the problem is
to determine which small deformations of $h$ arise as boundary values
of conformally compact ASD Einstein metrics near $g$. If $X$ is the
$4$-ball, and $g$ is the hyperbolic metric, we are in the realm of
LeBrun's positive frequency conjecture \cite{LeB91}, proved by Biquard in
\cite{Biquard}.  We shall see that in the present context of hyperk\"ahler
metrics, there is a similar positive frequency condition that needs to be satisfied by
deformations of the closed framing $\gamma$ if they are to arise as boundary values of
hyperk\"ahler deformations of $\omega$.

\subsection{Overview of results}

\subsubsection{Formal picture}

We begin with a formal, non-technical discussion of the set-up in
order to motivate the statements of our main theorems.
Let\footnote{Unless indicated otherwise, all functions, metrics, forms etc.~on $X$ are smooth up
  to (and including) the boundary of $X$.}
\begin{equation}\label{e1.10.2.16}
\cF = \{ \omega = (\omega_1,\omega_2,\omega_3) \in  \Omega^2(X)\otimes
\RR^3: \omega_1^2+\omega_2^2+\omega_3^2 >0\}.
\end{equation}
The `gauge group' $\cG_0$ of orientation-preserving diffeomorphisms of $X$ which are the
identity on $\p X$ acts on $\cF$ by pull-back.

The space of hyperk\"ahler triples $\cH$ on $X$ is defined as the subset
\begin{equation}\label{e2.10.2.16}
\cH = \{\omega \in \cF: \rd \omega =0,\; Q(\omega) =0\},
\end{equation}
where $Q$ is the $3\times 3$ matrix whose entries are
\begin{equation}\label{Q0}
Q(\omega)_{ij} 
=
\frac{\omega_i \wedge \omega_j}
{\frac{1}{3}(\omega_1^2 + \omega_2^2 +\omega_3^2)}
-
\delta_{ij}.
\end{equation}
This follows by combining \eqref{omegas_orthogonal} and
\eqref{hk_volume_form}. 

By linearizing $Q$ at a  hyperk\"ahler triple $\omega$, we see that
\begin{equation}\label{e21.12.8.15}
T_\omega \cH  = \{(\theta_1,\theta_2,\theta_3)\in
\Omega^2(X)\otimes\RR^3 : \rd\theta_i=0 \text{ and }
s^2_0(\theta_i,\omega_j)=0\},
\end{equation}
where $(\cdot,\cdot)$ denotes the pointwise inner product on
$2$-forms and $s^2_0(A)$ denotes the symmetric trace-free part of the
matrix $A$.

The group $\cG_0$ of diffeomorphisms acts on $\cH$ so we would like
to construct the moduli space $\cM = \cH/\cG_0$, and in particular to
understand its local structure. 

Continuing formally, we note that the tangent space to the
$\cG_0$-orbit through $\omega$ is the infinite-dimensional space
\[
\cB_\omega = \{ \cL_v \omega : v \in C^\infty(X,TX), \text{ $v$ vanishes on
} \p X\}.
\]
If $\omega$ is a closed triple, $\cL_v\omega = \rd(\iota_v\omega)$ by
Cartan's formula.  Therefore we define
\begin{equation}\label{e11.10.2.16}
L: C^\infty(X,TX) \rightarrow \Omega^2(X)\otimes
\RR^3,\quad Lv= \rd(\iota_v\omega).
\end{equation}
Denote by $L^*$ the formal adjoint of $L$. This operator
\begin{equation}\label{e12.10.2.16}
L^*: \Omega^2(X)\otimes \RR^3 \longrightarrow C^\infty(X,TX)
\end{equation}
is a composite
\begin{equation}\label{e13.10.2.16}
\Omega^2(X)\otimes \RR^3 \stackrel{\rd^*}{\longrightarrow}
\Omega^1(X)\otimes \RR^3 \stackrel{\pi}{\longrightarrow} \Omega^1(X)
\end{equation}
where $\pi$ is zeroth order, i.e.\  is given by a linear
bundle map.

Continuing formally, we should expect
\begin{equation}\label{e12.12.8.15}
T_{[\omega]}\cM =  T_\omega\cH/\cB_\omega = T_{\omega}\cH \cap \ker L^*
\end{equation}
and indeed we shall establish a suitable version of this statement in \S\S2--3.

Written this way, the infinitesimal geometry of the moduli space is
not so clear.  However, there are two subspaces of \eqref{e12.12.8.15}
that can be written down by hand, and it turns out that together 
these give all infinitesimal deformations.

The
first family consists of triples $\theta\in\cZ^2_-(X)\otimes\RR^3$ of closed anti-self-dual (ASD)
$2$-forms. Such a triple clearly satisfies the conditions
\eqref{e21.12.8.15}: it is closed by hypothesis, and 
satisfies $(\theta_i,\omega_j)=0$ because the self-dual
and anti-self-dual $2$-forms are pointwise orthogonal.  Being closed
and ASD also means that $\theta$ is coclosed, so it also lies in the
kernel of $L^*$ by  \eqref{e12.10.2.16} and \eqref{e13.10.2.16}.

The second family of deformations comes from diffeomorphisms which `move the boundary'.  More precisely,
suppose that $X \subset (X',\omega')$ is a domain in a larger
hyperk\"ahler 4-manifold and that $\omega$ is the restriction of the
hyperk\"ahler triple $\omega'$ from $X'$. Given a smooth map $f \colon
X \to X'$ which is a diffeomorphism onto its image, $f^*(\omega')$
gives a new hyperk\"ahler triple on $X$. On the infinitesimal level,
these give deformations of the form $\theta=\cL_v\omega = L(v)$ where $v$ is any
vector field on $X$, not necessarily vanishing along $\p X$.  By
naturality of the equations, $\theta$ 
automatically lies in $T_\omega\cH$; if we choose $v$ so that
$L^*Lv=0$, then $\theta$ will also lie in the `gauge-fixed' tangent
space \eqref{e12.12.8.15}.  In other words, if we define
\[
\cW = \{ v\in C^{\infty}(X,TX) : L^*Lv = 0\},
\]
then $L(\cW) \subset T_\omega\cH \cap \ker L^*$.

Observe that (in contrast to the case where $X$ is a compact manifold
without boundary) both of these families of deformations are
infinite-dimensional (cf.\ \S\ref{sect5}).

\subsubsection{Local structure of $\cM$}

Our first result makes the above statements precise, giving Banach
manifold structures on the infinite-dimensional spaces we have discussed.  Denote by $H^s$ the Sobolev
space of functions with $s$ derivatives in $L^2$.  We shall fix $s>4$
so that our Sobolev functions have good multiplicative properties.
Then we have
`finite-regularity' versions of the above spaces: $\cF^s$ is the space
of triples of $2$-forms with coefficients in $H^s$, $\cG^{s+1}_0$ is
the group of orientation-preserving diffeomorphisms which are the identity on $\p X$ given by mappings in $H^{s+1}$ and so
on.   In particular $\cG^{s+1}_0$ acts on $\cH^s$ and leads to the
regularity-$s$ moduli space $\cM^s=\cH^s/\cG^{s+1}_0$.  Our first result gives
information about the local structure of $\cM^s$ at a point $[\omega]$
represented by a {\em smooth} triple $\omega$.

\begin{thm}\label{smooth_moduli_space}
Let $\omega$ be a smooth hyperk\"ahler triple on $X$ and fix $s>4$.
\begin{enumerate}
\item[(i)] (Slice theorem.) A neighbourhood of $[\omega]$ in
  $\cF^s/\cG^{s+1}_0$ is homeomorphic to 
 a neighbourhood of $0$ in
$$
T_{\omega}\cF^s \cap \ker L^*=H^s(X,\Lambda^2\otimes \RR^3)\cap\ker L^*.
$$
\item[(ii)]  A neighbourhood of $[\omega]$ in $\cM^s =
  \cH^s/\cG_0^{s+1}$ is homeomorphic to a neighbourhood of $0$ in
\begin{equation}\label{e31.11.2.16}
[\cZ^2_-(X)\otimes \RR^3 \cap H^s]+ L(\cW^{s+1})  \subset
H^s(X,\Lambda^2\otimes \RR^3).
\end{equation}
Here $H^s(X,E)$ denotes the space of sections of the bundle $E$ with
coefficients in $H^s$ and $\cZ^2_-(X)$ is the space of closed anti-self-dual $2$-forms on $X$. \label{all_deformations}
\end{enumerate}
\end{thm}

In particular, for each (sufficiently high) degree of regularity $s$,
each smooth point $[\omega]$ of $\cM^s$ has a neighbourhood which is
homeomorphic to an open subset of a Banach space.  Note, however, that
 triples $\wh{\omega}$ near $\omega$ in the slice
$L^*(\wh{\omega}-\omega) =0$ are all smooth in the interior of $X$.

\subsubsection{The results of Bryant and Cartan}

Bryant's `thickening' result mentioned above suggests the formulation
of a natural boundary value problem for hyperk\"ahler metrics.
Namely:
\begin{prob}\label{pr1.11.2.16}
Given a closed framing $\gamma$ of $\Lambda^2_Y$, does there exist a
hyperk\"ahler triple $\omega$ on $X$ with $\iota^*\omega = \gamma$?
\end{prob}
We consider, and give an answer to, the following easier question:
\begin{prob}\label{pr2.11.2.16}
Given a hyperk\"ahler triple $\omega$ on $X$, with induced boundary
framing $\gamma$ of $\Lambda^2_Y$, which nearby closed framings of
$\Lambda^2_Y$ also arise from hyperk\"ahler triples on $X$?
\end{prob}
Before describing our results, let us discuss the relation between
these boundary value problems and the thickening result of Bryant
(which he attributes to Cartan).  A statement of this result is as follows:

\begin{thm}[Bryant, \protect{\cite[Theorem~1]{Bryant}}]\label{bryant.thm}
Let $Y$ be a closed 3-manifold. Given a closed framing $\gamma$ of $\Lambda^2_Y$ which is real-analytic for some analytic structure on $Y$ there is an essentially unique embedding $f \colon Y \to (X,\omega)$ into a hyperk\"ahler 4-manifold for which $f^*\omega = \gamma$. Here, ``essentially unique'' means that if $f' \colon Y \to (X',\omega')$ is another such embedding then there is an isometry $\phi$ from a neighbourhood of $f(Y) \subset X$ to a neighbourhood of $f'(Y) \subset X'$ such that $f' = \phi \circ f$. 
\end{thm}

\begin{rmk}\label{r1.12.2.16}
In \cite{Cartan}, Cartan states that such hyperk\"ahler metrics depend, locally and
modulo diffeomorphism, on two functions of three variables.  This is
obvious from the above description.  Each closed $2$-form is locally
determined by $2$ functions of $3$ variables, making $6$ functions of
$3$ variables in all. Dividing by (ambient) diffeomorphisms reduces
this by $4$ functions of $3$ variables, leaving two functions of $3$ variables.
\end{rmk}

A closed framing $\gamma$ of $\Lambda^2_Y$ determines
an induced metric on $Y$, which 
is fixed by the requirement that $\gamma$ be an {\em
  orthonormal} frame.  To see that such a metric exists and is unique,
plugging $\gamma$ into the canonical isomorphism
\begin{equation*}
\Lambda^2_Y = TY\otimes \Lambda^3_Y
\end{equation*} gives us a
framing $\alpha=(\alpha_1,\alpha_2,\alpha_3)$ of $TY\otimes
\Lambda^3_Y$.  Hence
$$\alpha_1\wedge\alpha_2\wedge\alpha_3 \in \Lambda^3
TY\otimes (\Lambda^3_Y)^3 = (\Lambda^3_Y)^2
$$
is non-zero at all points of $Y$.  Since $(\Lambda^3_Y)^2$ is a
trivial real line-bundle, there is a volume form
$\rd \mu_Y$ and a (unique) choice of sign such that $(\rd\mu_Y)^2 = \pm
\alpha_1\wedge\alpha_2\wedge\alpha_3$.  Using this volume form to
trivialize $\Lambda^3_Y$ we may now declare the $\alpha_i$ to be
orthonormal and we have our induced metric. 

However, a closed framing of $\Lambda^2_Y$ is strictly more
information than a metric.  
We can see this explicitly
through the Eguchi--Hanson and Taub--NUT spaces.  These provide a pair
of non-isometric hyperk\"ahler spaces $X$ and $X'$ in which we can
find isometric hypersurfaces $Y$ and $Y'$. (It is enough to pick $Y$
and $Y'$ to be suitable $\SU(2)$-orbits---further
details are provided in $\S$\ref{examples}, where we shall see that
the induced framings are different.)   

In general, a given
metric on $Y$ can admit many different closed orthonormal framings of
$\Lambda^2_Y$. Starting from one such framing $\gamma$, the others are
of the form $r ( \gamma)$ where $r \colon Y \to \SO(3)$ is such that
$\rd(r( \gamma)) = 0$. This is an underdetermined equation with an
infinite-dimensional space of solutions. (On the infinitesimal level,
it is analogous to prescribing the curl of a vector field on a
3-manifold.) At least when $r$ is non-constant, Bryant's theorem
applied to these different framings gives \emph{non-isometric}
hyperk\"ahler extensions of the same metric on $Y$ (assuming the
analyticity hypotheses are met).

These remarks are intended to explain that one's first guess, that the
boundary value of a hyperk\"ahler metric should be the induced metric
on the boundary, leads to an ill-posed boundary problem: the above
discussion shows that the same metric can bound non-isometric
hyperk\"ahler structures, whereas Theorem~\ref{bryant.thm} shows that
the formulation in terms of framings does not suffer from this problem.

As an aside, we have a further interpretation of a hyperk\"ahler
extension of a closed framing $\gamma$ of $\Lambda^2_Y$ as an
evolution equation, sometimes called an
``$\SU(2)$-flow''\footnote{Though this is not in any reasonable sense a parabolic equation.}.  If 
$\gamma_t$ is a 1-parameter family of closed framings of $\Lambda^2_Y$ then $\omega=\rd t\wedge *_t\gamma_t+\gamma_t$, where $*_t$ is the Hodge star on $Y$ determined by $\gamma_t$, is a hyperk\"ahler triple on 
$I\times Y$ for some interval $I\subseteq\RR$ if and only if
\begin{equation}\label{flow.eq}
\frac{\partial}{\partial t}\gamma_t=\rd *_t\gamma_t.
\end{equation}
This is sometimes interpreted as an evolution equation with $\gamma_0=\gamma$, which evolves amongst closed framings.  Thus we can also view 
our study of the boundary value problem as identifying initial conditions 
for which the evolution equation has good existence and convergence properties.  

\subsubsection{Boundary value problem for hyperk\"ahler metrics}

We shall now describe what we can say about Problem~\ref{pr2.11.2.16}.

As first order systems of equations are involved, it is not
reasonable to expect that all such nearby closed framings on $Y$ will bound
hyperk\"ahler triples on $X$;
instead, only those satisfying a {\em negative frequency condition}
will do so.  This is motivated by experience with boundary value
problems for Dirac operators, the prototypical example of which is the
$\dbar$-operator on the disc.  In this case, a given function 
$f:S^1\to \CC$ is only the boundary value of a function $u$ in the
disc with $\dbar u=0$ if all negative Fourier coefficients of $f$
vanish. 

Our results are cleanest in the case that $\omega$ induces {\em
positive mean curvature} on $Y$.   It is worth noting that if $\omega$ is a hyperk\"ahler triple on $X$ and $\gamma=\iota^*\omega$, the trace of the matrix of inner products 
$(\gamma_i,\rd *\gamma_j)$ is twice the mean curvature of $Y$ in $X$
(in fact, this matrix encodes the second fundamental form of $Y$).
 Denote by $\cM^s_+$ the
submoduli space of such hyperk\"ahler structures on $X$.  For
simplicity, suppose also that $H^1(Y)=0$.

The relevant Dirac operator for the hyperk\"ahler problem
turns out to be 
\begin{equation}\label{e21.11.2.16}
D : \Omega^1(X) \longrightarrow
 \Omega^0(X)\oplus \Omega^2_+(X),\quad D = \rd^* + \rd_+.
\end{equation}

On $Y$, with its induced Riemannian metric, we have the operator
\begin{equation}\label{e1.11.3.16}
D_Y: \Omega^0(Y)\oplus \Omega^1(Y) 
\longrightarrow \Omega^0(Y)\oplus \Omega^1(Y),
\;\;
D_Y = \begin{bmatrix} 0 & \rd^* \\ \rd & *\rd \end{bmatrix}.
\end{equation}
This is a self-adjoint operator of Dirac type, so if we define
\begin{equation}\label{e2.11.13.16}
H_\lambda = \ker(D_Y - \lambda),
\end{equation}
then $\dim H_\lambda < \infty$ and the set of $\lambda$ for which
$H_\lambda\neq 0$ is discrete and unbounded in both directions.
Define further
\begin{equation}
G_\lambda = H_\lambda \cap \ker(\rd^*).
\end{equation}
Clearly the set of $\lambda$ with $G_\lambda\neq 0$ is also discrete,
$\dim(G_\lambda)<\infty$ and it can also be shown that the set of
$\lambda$ with $G_\lambda\neq 0$ is unbounded in both
directions\footnote{We thank Dmitri Vassiliev for useful discussions
  on this point.}.

For any given $s>1/2$, define  $H^{s-1/2}_+(Y)$ to be the completion
of $\oplus_{\lambda>0} H_\lambda$ with respect to the $H^{s-1/2}$-norm and
$G^{s-1/2}_{-}(Y)$ to be the completion of $\oplus_{\lambda <0}
G_{\lambda}$ with respect to the same norm.  Then we have the
following result.

\begin{thm}\label{asd_boundary_values}
Let $\omega$ be a smooth hyperk\"ahler triple on $X$ inducing positive mean curvature on $Y$.  
Then the gauge-fixed tangent space
\begin{equation}\label{e11.11.2.16}
T_{[\omega]}\cM^s_+ = T_{\omega}\cH^s \cap \ker L^*
\end{equation}
is naturally isomorphic to the direct sum
\begin{equation}\label{e12.11.2.16}
\cH^2_{0,-}(X)\oplus \cH^2_-(X) \oplus G^{s+1/2}_-(Y)\otimes \RR^3 \oplus H^{s+1/2}_+(Y),
\end{equation}
where $\cH^2_{0,-}(X)$ and $\cH^2_-(X)$ are certain finite-dimensional spaces of
closed anti-self-dual $2$-forms whose dimensions depend only on the cohomologies of
$X$, $Y$ and $(X,Y)$.
\end{thm}

The spaces $\cH^2_-(X)$ and $\cH^2_{0,-}(X)$ are defined in
\eqref{H2pm.eq} and \eqref{H20pm.eq} before
Theorem~\ref{cohom.thm}.  Up to these finite-dimensional topological
pieces, we have an effective parameterization of the tangent
space of $\cM^s_+$ in terms of boundary data.  In a little more
detail, $G^{s+1/2}_-(Y)$ gives a parameterization of the
boundary values of the {\em exact} ASD $2$-forms on $X$: given $\alpha
\in G^{s+1/2}_-(Y)$, we find a unique $u\in H^{s+1}(X,\Lambda^1)$ such that $D u =
0$; then $\rd u \in H^s(X,\Lambda^2_-)$ is an exact ASD $2$-form.
Similarly if $\beta \in H^{s+1/2}_+(Y)$, there is a unique solution
$w\in H^{s+1}$ of $L^*Lw=0$ with boundary value $\beta$. This existence and uniqueness of $w$ is
true without the frequency condition, but the point is that with this
frequency condition imposed, $Lw$ cannot be a triple of ASD
$2$-forms and so this restriction allows us to replace the sum in 
Theorem~\ref{smooth_moduli_space}(ii) by a direct sum.

Theorem~\ref{asd_boundary_values} is in agreement with Cartan's count of
 degrees of freedom in the
local moduli space of hyperk\"ahler metrics
(Remark~\ref{r1.12.2.16}) in the following
sense. Discarding the finite-dimensional spaces in
\eqref{e12.11.2.16}, the space of negative frequency coclosed
$1$-forms should be counted as $2$ negative frequency functions on $Y$
(i.e.~of $3$ variables). We have a triple of such coclosed $1$-forms,
so $6$ negative-frequency functions. On the other hand
$H^{s+1/2}_+(Y)$ contributes $4$ positive frequency functions of $3$
variables. Since diffeomorphisms which move the boundary are given by
$4$ full functions of $3$ variables, subtracting this leaves us with
just $2$ negative-frequency functions on $Y$.   This agrees with
Cartan's count: for a global boundary value problem, one can only
prescribe `half the data' on the boundary as compared with the
thickening problem.

Another way of stating \eqref{e12.11.2.16} is that we have a
direct sum decomposition
\begin{equation*}
T_{[\omega]}\cM^s_+ = [\cZ^2_-(X)\otimes \RR^3 \cap H^s] \oplus L(\cW_+^{s+1}),
\end{equation*}
where we introduce the obvious notation $\cW_+^{s+1}$ for elements of
$\cW^{s+1}$ with positive frequency boundary values.  This 
improves the description of this tangent space in
\eqref{e31.11.2.16} and allows us to prove the following refinement
of Theorem~\ref{smooth_moduli_space}.

\begin{thm}\label{global_moduli_space} 
The moduli space $\cM_+$ of hyperk\"ahler triples on $X$ inducing
positive mean curvature on the boundary $Y$, modulo the action of
$\cG_0$, is a  
  Fr\'echet manifold with
  $$T_{[\omega]}\cM_+=[\cZ^2_-(X)\otimes\RR^3]\oplus L(\cW_+).$$
\end{thm}

\begin{rmk} The point here is that while the space on the RHS depends
  on $\omega$, these spaces are canonically isomorphic for triples
  $\omega$ and $\omega'$ in the same path component of $\cM_+$.
\end{rmk}

\subsubsection{Examples}

An immediate source of examples of hyperk\"ahler 4-manifolds with
boundary to which we may apply our theory arises from taking balls (or
other bounded domains) in gravitational instantons, by which we mean
complete hyperk\"ahler 4-manifolds.  The simplest examples are those
with an isometric $\SU(2)$-action, and we relate our approach to the
standard classification of these spaces in \S\ref{examples} and use
them to illustrate our main results.

We start our discussion with the flat metric on $\RR^4$.  There are two isometric actions of $\SU(2)$ on $\RR^4$ (corresponding to left-
and right-multiplication by unit quaternions).  
Consider first the $\SU(2)$ action which rotates the hyperk\"ahler triple.  This gives rise to a
standard framing of $S^3$ by closed $2$-forms.  On the other hand, the
Taub--NUT metric is a (non-flat) complete hyperk\"ahler metric on
$\RR^4$ with $\SU(2)$-action which rotates the triple, and depending
on a parameter $m>0$. In \S\ref{examples} we shall find a sphere in
Taub--NUT such that the induced framing is within $O(m^{-2})$ of the
standard round framing of $S^3$, and we shall be able to verify
explicitly that the difference between these framings is negative
frequency; this is consistent with Taub--NUT being a small deformation,
over this ball, of the flat metric.

Playing a similar game (now choosing the other $\SU(2)$ action on $\RR^4$), by comparing the flat metric with
Eguchi--Hanson, we are also able to construct explicit positive
frequency deformations of the induced framing; they cannot be filled
by hyperk\"ahler metrics over the ball precisely because the
Eguchi--Hanson metric lives on the $4$-manifold $T^*S^2$, and this
obstructs the problem of filling this particular positive frequency deformation.
Put another way, this provides an initial condition on $S^3$ for the
hyperk\"ahler evolution equation \eqref{flow.eq} which develops a
singularity. 
 
\subsection{Contents}

We begin in \S\ref{slice_theorem_section} with the proof of part (i)
of Theorem~\ref{smooth_moduli_space}, showing, roughly speaking, that
$\ker(L^*)$ gives a transverse slice to the action of the
diffeomorphism group $\cG_0$.  The rest of
Theorem~\ref{smooth_moduli_space} is proved in \S\S\ref{spinor.section}--\ref{dirac.background.section}. 

Our results on the moduli space $\cM_+$,
Theorems~\ref{asd_boundary_values} and \ref{global_moduli_space} follow in
\S\ref{sglobal_mod}.   Finally, in \S\ref{examples}, we
turn to concrete examples, coming from gravitational instantons with
isometric $\SU(2)$-actions. 

\subsection{Acknowledgements} The authors would like to thank
 Olivier Biquard, Robert Bryant, Claude LeBrun and Dmitri Vassiliev for useful
  conversations.  JF was supported by ERC consolidator grant 646649 ``SymplecticEinstein'' and FNRS grant MIS F.4533.15. JDL was supported by EPSRC grant EP/K010980/1.  This
  work was started when MAS was in receipt of a Leverhulme Research
  Fellowship and concluded as a visitor at MSRI in early 2016. The support of
  Leverhulme, MSRI and the NSF (grant number DMS-1440140) are gratefully acknowledged.

\section{Gauge fixing for the action of diffeomorphisms}\label{slice_theorem_section}

We begin by recalling some fundamental facts about the group of
diffeomorphisms, setting up notation along the way. (Details can be
found in \cite{Ebin, EbinMarsden}.) Let $X$ be a compact 4-manifold
with boundary $Y$ and let $\omega = (\omega_1, \omega_2, \omega_3)$ be a
hyperk\"ahler triple on $X$. We write $\cF = \Omega^2(X) \otimes \RR^3$
for the space of triples of 2-forms on $X$ and $\cG_0$ for the
group of orientation-preserving diffeomorphisms of $X$ which are the
identity on the boundary 
$Y$. The group $\cG_0$ acts on $\cF$ by pull-back. The goal in this
section is to find a slice for the action.  

To do so we will work with Hilbert space completions of $\cF$ and $\cG_0$. Given $s\geq 0$, we write $\cF^s$ for the Hilbert space of triples of 2-forms whose distributional derivatives are square-integrable up to order $s$. We can also talk of $H^s$-maps $X \to X$. When $s>3$,  $H^s$ embeds into $C^1$ and so an $H^s$-diffeomorphism makes sense: we write $\cG_0^s$ for the collection of orientation-preserving $H^s$-maps $X\to X$ which are also $C^1$-diffeomorphisms and which restrict to the identity on the boundary. The inverse of $\phi \in \cG_0^s$ is again of regularity $H^s$ and the composition of $H^s$-diffeomorphisms is in $H^s$. This makes $\cG^s_0$ into a topological group.

We next briefly recall the Banach manifold structure on $\cG^s_0$. Fix a Riemannian metric on $X$. (This metric is purely auxiliary and the manifold structure on $\cG_0$ turns out not to depend on this choice.) Write $\cV_0^s$ for the space of $H^s$ vector fields which vanish on the boundary. The geodesic exponential map is not defined on all tangent vectors, since geodesics can leave through the boundary. However, if $v \in \cV_0^s$ is sufficiently small in $H^s$, with $s>3$, then it is also bounded by $1/2$ say, in $C^1$. This means that for any $p \in X$, $|v(p)| \leq \frac{1}{2}d$ where $d$ is the distance of $p$ from $Y$. It follows that $\exp(v(p))$ still lies in $X$ and hence that exponentiating vectors to geodesics defines a map from a small neighbourhood of the origin in $\cV_0^s$ to $\cG_0^s$. One can check that this map is a homeomorphism onto a neighbourhood of the identity giving us a chart there. To define a chart near another point $\phi \in \cG_0^s$, we repeat the same construction, using $H^s$-sections of $\phi^*TX$. One then checks that these charts have smooth transition functions making $\cG_0^s$ into a Banach manifold. We refer to \cite{Ebin, EbinMarsden} for details. It is also shown there that right multiplication by a fixed diffeomorphism $\phi \in \cG_0^{s}$ gives a smooth map $\cG_0^s \to \cG_0^s$; this will be  important in what follows. (On the other hand, left multiplication is merely continuous, so that $\cG_0^s$ is a topological group, but not strictly speaking a Banach Lie group.)

The group $\cG^{s+1}_0$ acts on $\cF^s$ by pull-back. To find a slice for this action we follow the standard approach of taking the orthogonal complement to the infinitesimal group action. Given a vector field $v$ on $X$ which vanishes on $Y$, the infinitesimal action of $v$ on $\cF$ at $\omega$ is by Lie derivative, $\cL_v\omega$. We write $L$ for the map defined by $Lv = \cL_v\omega$ and $L^*$ for its formal adjoint. The main aim of this section is to prove the following, for which it is crucial that $\omega$ is smooth.

\begin{thm}\label{main.slice.thm}
Fix $s>4$. There exist constants $\epsilon, \delta >0$ such that for every $\hat{\omega}\in\cF^s(X)$ with $\|\hat{\omega}-\omega\|_{H^s} < \epsilon$ there exists a unique diffeomorphism $\phi\in \cG^{s+1}_0$ such that both $L^*(\phi^*\hat{\omega}-\omega)=0$ and $\| \phi^*\hat{\omega} - \omega\|_{H^s} < \delta$. In other words, the slice
\[
S_\delta = \{ \omega + \chi : \chi \in \cF^s,\ L^*\chi = 0,\ \| \chi\|_{H^s} < \delta \}
\]
meets every nearby orbit of $\cG^{s+1}_0$ exactly once.  
\end{thm}

This is the analogue of the Ebin--Palais slice theorem (see e.g.~\cite{Tromba}), which applies to diffeomorphisms acting on metrics (on a compact manifold without boundary) and we follow the outline of their proof closely. Accordingly, we omit the parts of the proof which are identical to Ebin--Palais.


Theorem~\ref{main.slice.thm} contains two assertions: existence and uniqueness. We begin with existence. The idea is to employ the implicit function theorem, but there is a subtlety because the action
\[
\cG_0^{s+1} \times \cF^s \to \cF^s
\]
is \emph{not} smooth. To see this note that the derivative at the point $(\id, \hat{\omega})$ should be given by the Lie derivative $\cL_v \hat{\omega}$. But if $\hat{\omega} \in H^s$ then $\cL_v\hat{\omega} \in H^{s-1}$ which is of insufficient regularity. \emph{The linearised action differentiates $\hat{\omega}$ whereas the full action doesn't.} This means that the existence part of the slice theorem does not follow from a simple application of the implicit function theorem (in spite of what one sometimes reads!) since this would need the action to be $C^1$. 

Instead, following Ebin, we proceed as follows. Write $\cE^s$ for the following subset of $\cG^{s+1}_0 \times \cF^s$:
\begin{equation}
\cE^s = \left\{ (\phi, \chi): L^*\left((\phi^{-1})^*\chi\right) = 0 \right\}.
\label{normal.bundle}
\end{equation}
There is a map $\cE^s \to \cF^s$ given by $F(\phi, \chi) = \phi^*\omega + \chi$. The idea is to apply the inverse function theorem to $F$, to show that it is a diffeomorphism between a neighbourhood of $(\id, 0) \in \cE^s$ and a neighbourhood of $\omega \in \cF^s$. Once this is done, it will follow that any $\hat{\omega}$ which is $H^s$-close to $\omega$ is of the form $\hat{\omega} = \phi^*\omega + \chi$ for $(\phi^{-1})^*\chi \in \Ker L^*$. Then $(\phi^{-1})^* \hat\omega - \omega \in \Ker L^*$ as required. One should think of $\cE^s$ as the normal bundle of the orbit of $\omega$ (pulled back to $\cG_0^{s+1}$) and $F$ as giving a tubular neighbourhood of the orbit. 

To push this argument through, there are three things which must be established:
\begin{itemize}
\item
$\cE^s$ is a Banach manifold. More precisely, we will show that $\cE^s \to \cG_0^{s+1}$ is a Banach vector bundle.
\item
The map $F$ is smooth.
\item
$\rd F$ at $(\id, 0)$ is an isomorphism.
\end{itemize}
We now explain how this works. The first step is to note that, whilst the whole action is not smooth, its restriction through a $C^\infty$ triple does give a smooth map. The proof is identical to that in Ebin--Palais, and so we omit it.  

\begin{lem}\label{smooth.action}
Let $s>3$. For a triple of smooth 2-forms $\omega$, the map $A \colon \cG_0^{s+1}\to \cF^s$ given by $A(\phi) = \phi^*\omega$ is a smooth map of Banach manifolds.
\end{lem}

We next explain why $\cE^s$, defined above in (\ref{normal.bundle}), is a smooth vector bundle over $\cG_0^{s+1}$. To do this (and still following Ebin) we define two operators. The first, $L_\phi \colon \cV^{s+1}_0 \to \cF^s$, is defined as follows. We use right multiplication to give a smooth trivialisation of the tangent bundle $T\cG_0^{s+1} \cong \cG_0^{s+1} \times \cV_0^{s+1}$. Then we set $L_\phi = \rd A_\phi$ to be the derivative of the action $A \colon \cG^{s+1}_0 \to \cF^s$ at $\phi \in \cG_0^{s+1}$ with respect to this trivialisation. 


The second operator, $L^*_\phi \colon \cF^s \to \cV^{s-1}$, is defined
to be the formal $L^2$-adjoint of $L_\phi \colon \cV^{s+1}_0 \to
\cF^s$, defined with respect to the metric determined by the
hyperk\"ahler triple $\phi^*\omega$. In particular, $L_{\id}^*=L^*$ is
precisely the operator appearing in the statement of the Slice
Theorem~\ref{main.slice.thm}. (Note that, as is standard with
boundary value problems, whilst $L_\phi$ is defined on vector fields
vanishing on the boundary, its adjoint $L_\phi^*$ will in general take
values in arbitrary vector fields.) 

We need two results concerning these operators. Again the proofs are identical to those in Ebin--Palais and so we omit them. 


\begin{lem}\label{Lphi-Lphi-star}
We have the following formulae for $L_\phi$ and $L_\phi^*$:
\[
L_\phi = \phi^* \circ L \circ \phi_*^{-1},\quad
L_\phi^* = \phi_* \circ L^* \circ (\phi^{-1})^*.
\]
In particular, $\ker L_\phi^* = \phi^* \ker L_{\id}^*$.
\end{lem}

\begin{lem}\label{Lphi.smooth}
Let $s>3$. The operators $L_\phi \in B(\cV^{s+1}_0,\cF^s)$ and $L_\phi^* \in B(\cF^s, \cV^{s-1})$ depend smoothly on $\phi \in \cG_0^{s+1}$. 
\end{lem}
%
 
So, by Lemmas~\ref{Lphi-Lphi-star} and~\ref{Lphi.smooth}, $\cE^s$ is the kernel of the smooth bundle map $\cG_0^{s+1} \times \cF^s \to \cV^{s-1}$, given by $(\phi, \chi) \mapsto L^*_\phi(\chi)$. This alone is not enough to ensure that $\cE^{s}$ is a vector bundle (even for finite rank bundles this can fail, since the dimension of the kernel can jump). To prove that $\cE^s$ is a Banach vector bundle we will use the following result. This is standard and so we do not give the proof.

\begin{lem}\label{bundle.split.lem}
Let $V$ be a Banach space and $M$ a Banach manifold. Suppose for each $x\in M$ we have a projection  $P_x\colon V\rightarrow V$, so that the map 
\[
M \rightarrow B(V), \quad
 x\mapsto P_{x}
 \]
is smooth. Then $\im P:=\cup_{x\in M} \im(P_x)$ and $\Ker P:=\cup_{x\in M} \Ker(P_x)$ are smooth sub-bundles of $M \times V$ and there is a splitting $M \times V=\im(P)\oplus \Ker(P)$. 
\end{lem}

%
%

We will ultimately apply Lemma \ref{bundle.split.lem} to the projection onto $\Ker L_\phi^*$ along $\im L_\phi$. At this point some parts of the proofs are minor modifications of those in Ebin--Palais and so we give the details. The first step is to show that $L = L_{\id}$ is invertible.

\begin{lem}\label{L.inj}
The map $L \colon \cV_0^{s+1} \to \cF^s$ is injective.
\end{lem}
\begin{proof}
If $Lv =  0$, then $v$ is a Killing field for the metric defined by $\omega$, but $v$ vanishes on $Y$ and so must also vanish on $X$.
\end{proof}

We next need a concise formula for $L^*$. We start with the adjoint of the map $TX \to T^*X \otimes \RR^3$ given by $v \mapsto \iota_v \omega$. The hyperk\"ahler triple $\omega = (\omega_1, \omega_2, \omega_3)$ determines a triple of complex structures $J = (J_1, J_2, J_3)$ via the requirement that $g(J_j u,v) = \omega_j(u, v)$ for $j=1,2,3$.  Using these we define a map $T^*X \otimes \RR^3 \to T^*X$ by 
\begin{equation}
a=(a_1,a_2,a_3) \mapsto J \cdot a\defeq J_1a_1 + J_2a_2 + J_3a_3. 
\label{Jdot_definition}
\end{equation}
The following lemma is the result of a simple calculation.

\begin{lem}
The map $a \mapsto (J\cdot a)^\sharp$ is adjoint to $v \mapsto \iota_v\omega$ (where $\alpha^\sharp$ is the vector metric-dual to the 1-form $\alpha$).
\end{lem}

\begin{lem}\label{Lstar}
We have $L^*\eta = (J \cdot \rd^* \eta)^\sharp$. In other words, if $v \in \cV_0$ and $\eta \in \cF$ then
\[
\langle Lv, \eta \rangle_{L^2} = \langle v,  (J \cdot \rd^* \eta)^\sharp \rangle_{L^2}
\]
\end{lem}
\begin{proof}
This is a direct calculation. Suppose that $v$ is a smooth vector field on $X$ vanishing on the boundary $Y$. Then
\begin{equation}\label{e1.3.2.17}
\langle Lv, \eta\rangle_{L^2}
	=
\int_X \rd(i_v\omega)\wedge *\eta
	=
\int_Y i_v\omega\wedge *\eta+\int_X i_v\omega\wedge \rd*\eta
	=
\langle v,  (J \cdot \rd^* \eta)^\sharp \rangle_{L^2}
\end{equation}
where we have used that $\cL_v \omega = \rd (i_v \omega)$ (since $\rd \omega=0$) and that the boundary integral vanishes since $v$ is zero on $Y$.
\end{proof}

We now prove a Hodge decomposition for $L$ and $L^*$.

\begin{prop}\label{tgt.split.prop}~
\begin{enumerate}
\item
The map $L^*L \colon \cV_0^{s+1} \to \cV^{s-1}$ is an isomorphism. 
\item 
There is a splitting $\cF^s= \im L\oplus\Ker L^*$.
\item
The map $P=L \circ (L^*L)^{-1} \circ L^*$ is the projection onto $\im L$ with respect to this splitting.
\end{enumerate}
\end{prop}

\begin{proof}
We start with point 1. The operator $L$ is a first order differential operator with symbol
\[
\sigma_L(\xi,v)=\xi\wedge i_v\omega,
\]
which is injective.  In fact, since $i_v\omega_i=(J_iv)^{\flat}$ form an orthogonal triple of 1-forms, we have
\[
|\sigma_L(\xi,v)|^2
	=
		\sum_{j=1}^3|\xi|^2|J_jv|^2-\langle\xi,i_v\omega_j\rangle^2
	=
		3|\xi|^2|v|^2-\sum_{j=1}^3\langle\xi,i_v\omega_j\rangle^2\geq 2|\xi|^2|v|^2.
\]
Thus $\langle\sigma_{L^*L}(\xi,v),v\rangle \geq 2|\xi|^2|v|^2$, meaning $L^*L$ is a strongly elliptic operator in the sense of \cite[Chapter 5 (11.79)]{Taylor}. It follows that the Dirichlet problem for $L^*L$ is regular in the sense of \cite[p.454]{Taylor} by \cite[Chapter 5 Proposition 11.10]{Taylor}.  We may thus apply \cite[Chapter 5 Proposition 11.16]{Taylor} to conclude that the map
\begin{align}
\cV^{s+1}(X)&\rightarrow \cV^{s-1}(X)\oplus \cV^{s+\frac{1}{2}}(Y)\nonumber\\
v&\mapsto (L^*Lv,v|_Y)
\label{bdry.map.eq}
\end{align}
is Fredholm, where $\cV(Y)$ denotes sections of $TX|_Y$.  

Now if $v\in\cV_0^{s+1}$ with $L^*Lv=0$ then
\[
0=\langle L^*Lv,v\rangle_{L^2}=\|Lv\|^2_{L^2}.
\]
Since $L$ is injective (Lemma~\ref{L.inj}) it follows that $L^*L$ is
injective on $\cV_0^{s+1}$.   The operator $L^*L$ is formally
self-adjoint, and from the formula in the proof of Lemma~\ref{Lstar}
$v|Y=0$ is a self-adjoint boundary condition.  Indeed, if  $\eta =
Lw$ in \eqref{e1.3.2.17}, for another vector field $w$, we obtain
\[
\langle Lv, Lw\rangle_{L^2} - \langle v,  L^*Lw \rangle_{L^2}
	=
\int_Y i_v\omega\wedge *Lw
\]
and so by skew symmetrizing,
\[
\langle L^*L v, w\rangle_{L^2} - \langle v,  L^*Lw \rangle_{L^2}
	=
\int_Y \left(i_v\omega\wedge *Lw  - i_w\omega\wedge *Lv\right).
\]
Since $v|Y=0$ (as an element of $TX|Y$) implies that
$\iota^*(\iota_v\omega)=0$, we see that $v|Y=0$ is a self-adjoint
boundary condition for the operator $L^*L$ on $X$.   Hence the index
of (\ref{bdry.map.eq}) is zero, which therefore means it is surjective
as well, and so $L^*L \colon \cV_0^{s+1} \to \cV^{s-1}$ is an isomorphism.

We next turn to the splitting claimed in point~2. The sum is clearly direct because if $\eta=Lv$ for $v\in\cV_0^{s+1}$ and $L^*\eta=0$ then $L^*Lv=0$. To show the sum spans, let $\eta\in\cF^s$; we must find $v\in\cV_0^{s+1}$ such that $\eta-Lv$ is in $\Ker L^*$. This amounts to solving $L^*Lv=L^*\eta$ for $v\in\cV_0^{s+1}$ which we can do by the surjectivity of (\ref{bdry.map.eq}).

Finally, for point~3, note $P^2=P$, $P(L v) = Lv $ and that $P$ vanishes on $\Ker L^*$.
\end{proof}

\begin{cor}
Let $s>3$. For all $\phi \in \cG_0^{s+1}$, the following are true.
\begin{enumerate}
\item
The map $L^*_\phi L_\phi \colon \cV_0^{s+1} \to \cV^{s-1}$ is an isomorphism.
\item
There is a splitting $\cF^s = \im L_\phi \oplus \Ker L_\phi^*$.
\item
The map $P_\phi = L_\phi \circ (L_\phi^*L_\phi)^{-1} \circ L_\phi^*$ is the projection onto $\im L_\phi$ with respect to this splitting.
\end{enumerate}
\end{cor}
\begin{proof}
Recall that Lemma~\ref{Lphi-Lphi-star} asserts that
\[
L_\phi = \phi^* \circ L \circ \phi_*^{-1},\quad
L_\phi^* = \phi_* \circ L^* \circ (\phi^{-1})^*.
\]
It follows that $L_\phi^*L_\phi = \phi_* \circ  L^*L \circ \phi_*^{-1}$. Note that when decomposed in this fashion one must be careful to keep track of regularity since, for example, $\phi_*^{-1} \colon \cV_0^{s+1} \to \cV_0^s$. Nonetheless all three maps in this composition are injective and so the same is true of $L_\phi^*L_\phi$. Similarly to solve the equation $L_\phi^*L_\phi (v) = w$, with $w \in \cV^{s-1}$, one first solves $L^*Lu = \phi^{-1}_*w$ for $u \in \cV_0^s$ (by invertibility of $L^*L \colon \cV_0^s \to \cV^{s-2}$) and then sets $v = \phi_*u$. Now a~priori $v \in \cV_0^{s-1}$ but, since $L^*_\phi L_\phi(v) = w$, elliptic regularity ensures that $v \in \cV_0^{s+1}$ after all, proving surjectivity of $L^*_\phi L_\phi$. The other points follow exactly as before in the proof of Proposition~\ref{tgt.split.prop}.
\end{proof}

We have now justified the three key points mentioned in the introduction, namely:

\begin{itemize}
\item 
$\cE^s \to\cG^{s+1}_0$ is a Banach vector bundle. This follows from Lemma~\ref{bundle.split.lem} and the fact that the projection onto $\im L_\phi$ depends smoothly on $\phi$. This is because $P_\phi = L_\phi \circ (L_\phi^*L_\phi)^{-1}\circ L_\phi^*$ and each operator in this composition is smooth in $\phi$ (by Lemma~\ref{Lphi.smooth}).
\item
The map $F \colon \cE^s \to \cF^s$ given by $F(\phi, \chi) = \phi^*\omega + \chi$ is smooth, by Lemma \ref{smooth.action}.
\item
Its derivative at $(\id, 0)$ is given by
\[
\rd F \colon \cV^{s+1}_0 \oplus \Ker L^* \to \cF^s,\quad
\rd F(v ,\eta) = Lv + \eta
\]
which is an isomorphism, by Proposition~\ref{tgt.split.prop}.
\end{itemize}

The existence part of Theorem~\ref{main.slice.thm} now follows, just as in Ebin--Palais, by an application of the implicit function theorem. Accordingly, we state the result without writing the details.

\begin{thm}\label{existence}
Fix $s>3$. There are constants $C, \epsilon ,\delta >0$ and an open neighbourhood $U$ of $\id \in \cG^{s+1}_0$ such that if $\|\hat{\omega} - \omega\|_{H^s} < \epsilon$ then there exists a unique diffeomorphism $\phi \in U$ such that both $L^*(\phi^*\hat{\omega}- \omega) = 0$ and $\| \phi^*\hat{\omega} - \omega\|_{H^s} < \delta$. Moreover, in this case $\| \phi^*\hat{\omega} - \omega\|_{H^s} < C \| \hat{\omega} - \omega\|_{H^s}$.
\end{thm}

We next turn to uniqueness. The crux is to show that the action of $\cG_0^{s+1}$ on $\cF^s$ is proper, in a certain sense. Before we state this result, we need a preliminary definition.

\begin{dfn}
A triple of 2-forms $(\omega_1, \omega_2, \omega_3)$ on a 4-manifold is called \emph{definite} if there is a nowhere vanishing 4-form $\Omega$ such that the $3\times 3$-matrix-valued function $(\omega_i \wedge \omega_j)/\Omega$ is positive definite. 
\end{dfn}

A hyperk\"ahler triple is an example of a definite triple. Note that definiteness is an open condition in $\cF^s$, as long as $s$ is large enough that Sobolev multiplication holds. Given any definite triple $(\omega_1, \omega_2, \omega_3)$ the wedge product is definite on $\langle \omega_1,\omega_2,\omega_3 \rangle$ and hence there is a unique conformal class making the $\omega_j$ self-dual. One can then specify a metric in this conformal class by taking the volume form to be $\frac{1}{6}(\omega_1^2 + \omega_2^2 + \omega_3^2)$ (c.f.~\eqref{hk_volume_form}). In this way we canonically associate a Riemannian metric to every definite triple. (Of course, when the triple is hyperk\"ahler, this metric is the obvious one.)  For more details on definite triples see \cite{Donaldson}.

We are now ready to prove that the action of $\cG_0^{s+1}$ on $\cF^s$ is proper, at least when restricted to definite triples.   

\begin{thm}\label{proper.action}
Fix $s>4$. Let $\omega_n \in \cF^s$ be a sequence of definite triples and $\phi_n \in \cG^{s+1}_0$ a sequence of diffeomorphisms. Suppose that $\omega_n$ converges in $H^s$ to a definite triple $\omega$ and that $\phi_n^*\omega_n$ converges in $H^s$ to a definite triple $\hat{\omega}$. Then there is a subsequence of the $\phi_n$ which converges in $\cG^{s+1}_0$ to a diffeomorphism $\phi$. Moreover, $\phi \colon (X,\hat{g}) \to (X,g)$ is an isometry where $g$ and $\hat{g}$ are the Riemannian metrics associated to $\omega$ and $\hat{\omega}$ respectively .
\end{thm}

This is a direct analogue of---and follows immediately from---a theorem of Ebin--Palais for the action of diffeomorphisms on Riemannian metrics (cf.~\cite{Tromba}). (As an aside, the lower bound on $s$ is necessary for the proof of Ebin--Palais which uses Sobolev multiplication at a certain point.) 
 
\begin{proof}
Write $g_n$, $g$ and $\hat{g}$ for the Riemannian metrics corresponding to the definite triples $\omega_n, \omega$ and $\hat{\omega}$ respectively. We have that $g_n \to g$ and $\phi^*_n g_n \to \hat{g}$ in $H^s$. Now the result of Ebin--Palais gives a subsequence of the $\phi_n$ which converges in $\cG^{s+1}_0$ to a diffeomorphism $\phi$ satisfying $\phi^*g = \hat{g}$. \end{proof}
 
The full Slice Theorem~\ref{main.slice.thm} now follows from Theorems~\ref{existence} and~\ref{proper.action}, in identical fashion to Ebin--Palais's orignal slice theorem.

\section{The hyperk\"ahler equation modulo diffeomorphisms}
\label{spinor.section}

The goal of this section is to gauge fix the hyperk\"ahler equation in
order to be able to apply elliptic theory. The main result is
Theorem~\ref{regularity} below, which shows that the moduli space of
all hyperk\"ahler triples up to diffeomorphism is locally homeomorphic
to the zero locus of a non-linear  operator with certain ellipticity
properties.  

There are complications in arriving at Theorem~\ref{regularity} which come from the fact that there are two competing notions of gauge. The first is the differential condition of the previous section, coming from the action of diffeomorphisms on 2-forms. This has the advantage that triples of 2-forms can always be put in ``differential gauge'' by the Slice Theorem~\ref{main.slice.thm}. It does not, however, lead to an elliptic equation. The other kind of gauge fixing arises when one parametrises \emph{cohomologous} triples of 2-forms by triples of 1-forms $a$ via $\omega + \rd a$. This leads naturally to an \emph{algebraic} condition on $a$ and with this gauge imposed the hyperk\"ahler equation becomes genuinely elliptic. The problem, however, is that it is not in general possible to put a triple $a$ in ``algebraic gauge''  via the action of diffeomorphisms. The proof of Theorem~\ref{regularity} involves the interaction between these two notions of gauge.
 
\subsection{A non-linear Dirac equation}\label{nonlinear.Dirac}

Our starting point is the following formulation of hyperk\"ahler metrics in terms of triples of 2-forms (cf.~\cite{Donaldson,Hitchin}). The lemma is standard and accordingly we only sketch the proof.

\begin{lem}\label{l1.12.8.15}
Let $X$ be a 4-manifold and $(\omega_1, \omega_2, \omega_3)$ a triple of closed 2-forms on $X$. Suppose that 
\[
\omega_i \wedge \omega_j = \delta_{ij} \mu
\]
for some nowhere vanishing 4-form $\mu$. Then $X$ carries a hyperk\"ahler metric $g$ which is  characterised by the fact that the $\omega_i$ are all self-dual and the volume form is given by $\mu = 2 \rd V_g$. 
\end{lem}
\begin{proof}[Sketch of proof]
Since the wedge product is definite on the sub-bundle $\langle \omega_i \rangle$ of $\Lambda^2$ spanned by the forms $\omega_i$, there is a unique conformal class for which the $\omega_i$ are all self-dual. Choosing $\rd V_g = \mu/2$ determines a metric in this conformal class. The $\omega_i$ now give a metric trivialisation $\Lambda^2_+ \cong X \times \RR^3$ of the bundle  of self-dual 2-forms. Under this identification, the product connection preserves the metric in $\Lambda^2_+$ and, since the $\omega_i$ are closed, it is also torsion-free. It follows that the product connection is identified with the Levi-Civita connection in $\Lambda^2_+$, which is thus flat with trivial holonomy;  this is one characterisation of a hyperk\"ahler metric. 
\end{proof}

As  mentioned in the introduction, $\mu$ can be recovered from the symplectic forms via $\mu = \frac{1}{3}\sum \omega_i^2$. This means that hyperk\"ahler triples are exactly those triples of symplectic forms, all inducing the same orientation, solving the equation $Q(\omega) = 0$ where $Q(\omega)$ is the symmetric trace-free $3\times 3$-matrix valued function defined by
\begin{equation}
Q(\omega)_{ij} 
=
\frac{\omega_i \wedge \omega_j}
{\frac{1}{3}(\omega_1^2 + \omega_2^2 +\omega_3^2)}
-
\delta_{ij}.
\label{Q}
\end{equation}

Linearising $Q$ at a hyperk\"ahler triple $\omega = (\omega_1, \omega_2, \omega_3)$, we see that infinitesimal hyperk\"ahler deformations of $\omega$ are given by triples $\theta = (\theta_1, \theta_2, \theta_3)$ of closed 2-forms which lie in the kernel of the operator $P \colon \Lambda^2 \otimes \RR^3 \to S^2_0\RR^3$ (i.e., taking values in symmetric trace-free endomorphisms of $\RR^3$) defined by
\begin{equation}\label{P}
P(\theta)_{ij} 
= 
\frac{1}{2}(\theta_i, \omega_j) + \frac{1}{2}(\omega_i, \theta_j)
-
\frac{1}{3}\delta_{ij}\sum_{k=1}^3 (\theta_k, \omega_k).
\end{equation}
Here $(\theta_i ,\omega_j)$ etc.\ denote pointwise inner products. (To obtain this formula, recall that dividing by the volume form converts wedge products with self-dual 2-forms into inner products.) The operator $P$ can be written more succinctly by identifying $\RR^3 \cong \Lambda^2_+$ via $\omega$. Then $P$ is the map $\Lambda^2 \otimes \Lambda^2_+ \to S^2_0 \Lambda^2_+$ given by $P(\theta) = s^2_0(\theta_+)$, the projection onto the trace-free symmetric part of the self-dual component of $\theta$ in $\Lambda^2_+ \otimes \Lambda^2_+$.

We next consider infinitesimal deformations of $\omega$ which fix the cohomology class. These correspond to $\theta = \rd a$ for $a \in \Omega^1 \otimes \RR^3$ a triple of 1-forms which solve $P(\rd a) = s^2_0(\rd_+a) = 0$. There is ambiguity in the choice of $a$ with $\rd a = \theta$ fixed which we can reduce by requiring that $a$ is in ``Coulomb gauge'', 
$\rd^*a=0$.  
Such an $a$ can always be found (see Lemma~\ref{coulomb.gauge} below) but there is still redundancy in this parametrisation; there are many different solutions $a$ to $\rd a = \theta$ with $\rd^*a = 0$. Indeed on a manifold with boundary they form an infinite-dimensional space. Lemma~\ref{coulomb.gauge} shows how to cut this down to a space of dimension $b^1(X)$ by imposing appropriate boundary conditions. Before discussing this, we  look at the ``Coulomb gauge fixed'' operator $D(a) = (P(\rd a), \rd^*a)$ whose kernel parametrises infinitesimal cohomologous hyperk\"ahler deformations of $\omega$. 

As written, $D$ is a differential operator $D \colon \Omega^1 (\RR^3) \to C^{\infty}(S^2_0\RR^3 \oplus \RR^3)$ between sections of bundles of different ranks and so cannot be elliptic. This is to be expected because of the action of vector fields: given a vector field $v$, the triple $\cL_v\omega$ gives an infinitesimal hyperk\"ahler deformation of $\omega$ and so must lie in the kernel of $P$. Since $\cL_v \omega = \rd (i_v \omega)$, this suggests that on the level of 1-forms we should work orthogonal to triples of the form $i_v \omega$, i.e., consider $a$ with 
\begin{equation}\label{Jdota.eq}
J \cdot a =0
\end{equation}
 (where $J\cdot a$ is defined in (\ref{Jdot_definition})). Notice that this is an \emph{algebraic} condition, and is not the same as the \emph{differential} gauge fixing condition $L^*(\rd a) = 0$ of \S\ref{slice_theorem_section}. 

The advantage of \eqref{Jdota.eq} is that it leads directly to an elliptic operator; as we will show shortly, when suitably interpreted in this way $D$ is a Dirac operator. The disadvantage of \eqref{Jdota.eq} is that it cannot be imposed by acting via diffeomorphisms. The problem occurs at the boundary. On the infinitesimal level, to put $a$ in ``algebraic gauge'', one must solve $J \cdot (i_v \omega +a )= 0$, which amounts to $v = \frac{1}{3}(J\cdot a)^\flat$. For arbitrary $a$ this vector field will not vanish on the boundary and so the action of $\cG_0$ is not sufficient to ensure a given triple satisfies \eqref{Jdota.eq}.  

Nonetheless, understanding the restriction of $D$ to those $a$ with $J\cdot a=0$ will be crucial in the sequel. The most efficient way to proceed is via spinors. Write $S_+, S_- \to X$ for the positive and negative spin bundles of $X$ and $S^m_\pm$ for the $m^\text{th}$ tensor product of $S_{\pm}$. In what follows we will only ever encounter tensor products $S^m_+\otimes S^n_-$ with an even number of factors, $m+n = 2k$, and so the question of whether or not $X$ is spin can safely be ignored. Moreover, when $m+n =2k$ this tensor product carries a real structure and we will write $S^m_+\otimes S^n_-$ to mean the \emph{real} locus of this bundle, a real vector bundle of rank $(m+1)(n+1)$. 

We begin by recalling, without proof, some spinorial isomorphisms
(cf.~\cite{AHS}). 
\begin{lem}\label{spin_identities}
There are the following natural isomorphisms of vector bundles:
\begin{itemize}
\item
$S_+ \otimes S_- \cong TX \cong \Lambda^1$;
\item
$S^2_+ \cong \Lambda^2_+$;
\item
$S_+ \otimes S^m_+ \cong S^{m+1}_+ \oplus S^{m-1}_+$;
\item
$S^2_0(S^2_+) \cong S^4_+$,
\end{itemize}
where in the last isomorphism, $S^2_0(S^2_+)$ denotes trace-free symmetric endomorphisms of $S^2_+$.
\end{lem}

\begin{cor}\label{spinor_domain_range}
Let $(X, \omega)$ be a hyperk\"ahler 4-manifold. Using the hyperk\"ahler triple to identify $\Lambda^2_+ \cong \RR^3$, there are isomorphisms
\begin{eqnarray}
\Lambda^1 \otimes \RR^3 
	&\cong&
		(S_- \otimes S_+) \oplus (S_- \otimes S^3_+),
		\label{splitting_domain}\\
S^2_0(\RR^3) \oplus \RR^3
	&\cong&
		S_+ \otimes S^3_+.
		\label{splitting_range}
\end{eqnarray}
Moreover, the first summand in (\ref{splitting_domain}) is identified with triples of the form $\iota_v\omega$ where $v$ is a vector field whilst the second summand is identified with triples $a$ such that $J \cdot a = 0$. 
\end{cor}
\begin{proof}
The isomorphisms follow from Lemma~\ref{spin_identities}. To prove the last claim, note that the map $v \mapsto \iota_v \omega$ from $TX \to \Lambda^1 \otimes \RR^3$ is $\SU(2)$-equivariant under the natural action of $\SU(2)$ on $TX, \Lambda^1$ and $\RR^3 \cong \Lambda^2_+$. It follows that the image of this map agrees with the first summand in (\ref{splitting_domain}) by Schur's Lemma. Finally, since $a \mapsto (J \cdot a)^\sharp$ is the adjoint of $v \mapsto i_v \omega$, the second summand in (\ref{splitting_domain}) is identified with solutions to $J \cdot a = 0$. 
\end{proof}

\begin{prop}\label{D_is_Dirac}
Let $(X,\omega)$ be a hyperk\"ahler 4-manifold. On restriction to sections of the sub-bundle $S_- \otimes S^3_+ \subset \Lambda^1 \otimes \RR^3$, and under the isomorphisms of Corollary~\ref{spinor_domain_range}, the operator $D(a) = (s^2_0(\rd_+a), \rd^*a)$ is identified with the negative Dirac operator coupled to the Levi-Civita connection on $S^3_+$:
\[
\cD \colon C^{\infty}(S_- \otimes S^3_+) \to C^{\infty}(S_+ \otimes S^3_+).
\]
\end{prop}
\begin{proof}
We start with the standard fact that the operator $\rd^* + \rd_+ \colon \Omega^1 \to \Omega^0 \oplus \Omega^2_+$ is a Dirac operator. Namely, under the isomorphisms $\Lambda^1 \cong S_- \otimes S_+$ and $\RR \oplus \Lambda^2_+ \cong S_+ \otimes S_+$, $\rd^* + \rd_+$ is identified with the negative Dirac operator coupled to the Levi-Civita connection on $S_+$:
\[
\cD_1 \colon C^{\infty} (S_- \otimes S_+) \to C^{\infty}(S_+ \otimes S_+).
\]
(For a proof of this, see, for example, \cite{AHS} where they consider $\Lambda^2_-$ rather than $\Lambda^2_+$ but the idea is the same.) Next, we couple this Dirac operator to the bundle $S^2_+ \cong \RR^3$, which is flat since $X$ is hyperk\"ahler. This means that on \emph{triples} of 1-forms, the operator $\rd^* + \rd_+$ is again identified with a negative Dirac operator, this time coupled to the Levi-Civita connection on $S_+ \otimes S^2_+$:
\[
\cD_2 \colon C^{\infty}(S_- \otimes S_+ \otimes S^2_+) \to C^{\infty}(S_+ \otimes S_+ \otimes S^2_+).
\]
Finally, the following decompositions are parallel with respect to the Levi-Civita connection:
\begin{eqnarray*}
S_- \otimes S_+ \otimes S^2_+ 
	& \cong &
		(S_- \otimes S^3_+) \oplus (S_- \otimes S_+);\\
S_+ \otimes S_+ \otimes S^2_+
	& \cong &
		(S_+ \otimes S^3_+) \oplus (S_+ \otimes S_+).
\end{eqnarray*}
It follows that the restriction of $\cD_2$ to $C^{\infty}(S_-\otimes S^3_+)$ maps into $C^{\infty}(S_+\otimes S^3_+)$ where it agrees with the negative Dirac operator coupled to the Levi-Civita connection on $S^3_+$ as claimed.
\end{proof}

\begin{cor}\label{nonlin.Dirac}
The map $F \colon C^{\infty}(S_- \otimes S^3_+) \to C^{\infty}(S_+ \otimes S^3_+)$ given by
\[
F(a) = Q(\omega+\rd a) + \rd^*a
\]
is a non-linear Dirac operator, whose zeros define hyperk\"ahler triples. (Here, we use Corollary~\ref{spinor_domain_range} to identify the domain of $F$ with the subspace of triples in $\Omega^1(X)\otimes\RR^3$ satisfying \eqref{Jdota.eq}  and the range of $F$ with $C^{\infty}(X,S^2_0\RR^3 \oplus \RR^3)$.)
\end{cor}

Another elementary calculation that will be used later concerns the
operators $\rd_+$ and $\rd^*$ on triples in the summand $S_-\otimes S_+$ in
\eqref{splitting_domain}.  

Let $\alpha$ be a $1$-form and let $\tau_i = J_i\alpha$ be the
corresponding triple of $1$-forms. Then
\begin{equation}\label{e1.30.1.16}
\rd^* \tau_i \in C^\infty(X,\RR^3)
\end{equation}
and
\begin{equation}\label{e2a.30.1.16}
\rd_+\tau_i \in \Omega^2_+(X)\otimes\RR^3 \cong
C^\infty(X,\RR)\oplus C^\infty(X,\RR^3) \oplus
C^\infty(X,S^2_0\RR^3).
\end{equation}
\begin{prop}\label{linearisation.vector.field}
Under the identification $\RR^3 = \Lambda^2_+$ by $\omega$,
$\rd^*\tau$ in \eqref{e1.30.1.16} is identified with $\rd_+\alpha$.

For $\rd_+\tau$, we have
\begin{equation}
s_0^2(\rd_+\tau) = 0
\end{equation}
while the $\RR \oplus \RR^3$ components are identified respectively
with $\rd^*\alpha$ and $\rd_+\alpha$. 
\end{prop}
\begin{proof}
We first consider the $s^2_0$-projection of the matrix 
\[
(\omega_i, \rd \tau_j) 
	= 
		\frac{\omega_i \wedge \rd (J_j \alpha)}{\rd V_g}
	=
		\frac{ \rd \left( \omega_i \wedge J_j \alpha \right)}{\rd V_g}.
\]
The complex structures $J_i$ and K\"ahler forms $\omega_i$ are related by $J_i\alpha = *(\alpha \wedge \omega_i)$ and so
\begin{equation}\label{d+tau.eq}
(\omega_i, \rd \tau_j) =\frac{\rd *(J_iJ_j \alpha)}{\rd V_g}.
\end{equation} 
Now the quaternion relations for the $J_i$ imply that the $s^2_0$-projection of this matrix vanishes and the $\RR$-component of $\rd_+\tau$ is $\rd^*\alpha$ as claimed. 

Meanwhile $\rd^*\tau_i = -*\rd*(J_i \alpha) = * \rd (\omega_i \wedge \alpha) = *(\omega_i \wedge \rd \alpha) =  (\omega_i,\rd \alpha)$.
\end{proof}
\begin{rmk}
Another way of stating the second part of Proposition \ref{linearisation.vector.field} is in terms of the matrix
$(\omega_i,\rd \tau_j) = \frac{1}{2}(\omega_i,\rd_+\tau_j)$: specifically, it says that the trace part is equal to $\rd^*\alpha$, the skew
part is equal to $\rd_+\alpha$ and the $s^2_0$ part is zero.
\end{rmk}

\subsection{Proof of Theorem~\ref{smooth_moduli_space}}
\label{s2.29.1.16}

Having considered cohomologous triples $\omega  + \rd a$, with the
additional conditions $\rd^*a = 0 = J \cdot a$, we now turn to the
general case.   We shall write down a smooth map $\cQ$ with domain
essentially triples $a$ of $1$-forms on $X$ with coefficients in the
Sobolev space $H^{s+1}$ such that: $\cQ$ is a submersion at any given
hyperk\"ahler triple; and $\cQ^{-1}(0)$ is precisely the set of all
hyperk\"ahler triples $\widehat{\omega}_i$ with $L^*
(\wh{\omega} - \omega)=0$.  Theorem~\ref{smooth_moduli_space} will
follow from this.

\subsubsection{Hodge theory on $X$}
\label{s1.29.1.16}
We begin by recalling some Hodge theory for manifolds with
boundary. The standard reference for this material is
\cite{Schwarz}.  As before, denote by $\iota$ the boundary inclusion
$\iota: Y \to X$.   Given $\alpha \in \Omega^p(X)$,  define forms on
$Y$ as follows:
\begin{equation}\label{e1.13.9.15}
\alpha_{\top} = \iota^*(\alpha) \mbox{ and }
\alpha_{\bot} = \iota^*(*\alpha).
\end{equation}
Use these to define boundary conditions for two spaces of harmonic forms:
\begin{align}
\cH^p_{\top} 
	&=
		\{ \alpha \in \Omega^p(X) : \rd \alpha = 0,\quad \rd^* \alpha = 0, \quad \alpha_{\top} =0\},\\
\cH^p_{\bot}
	&=
		\{ \alpha \in \Omega^p(X) : \rd \alpha = 0,\quad \rd^* \alpha = 0, \quad \alpha_{\bot} =0\}.\label{Neumann.eq}
\end{align}
Elements of $\cH^p_{\top}$ are called \emph{Hodge forms satisfying
  Dirichlet boundary conditions} and elements of $\cH^p_{\bot}$ are
called \emph{Hodge forms satisfying Neumann boundary conditions} (even
though the traditional Neumann condition involves a normal derivative,
unlike here). The Hodge theorem for manifolds with boundary, due to
Morrey--Friedrichs, is as follows.
\begin{thm}\label{hodge.boundary}
The inclusions $\cH^p_{\top} \to \Omega^p(X,Y)$ and $\cH^p_{\bot} \to
\Omega^p(X)$ induce isomorphisms 
\[
\cH^p_{\top} \cong H^p(X,Y), \quad \cH^p_{\bot} \cong H^p(X).
\]
\end{thm}
\noindent (Here $\Omega^p(X,Y)$ is the space of $p$-forms on $X$ which
restrict to $0$ on $Y$; these vector spaces give a complex under exterior
derivative and $H^p(X,Y)$ is the resulting cohomology group.) 

With this in hand, we can give a convenient parametrisation of closed 2-forms as follows.

\begin{lem}\label{coulomb.gauge}
Let $\theta$ 
be a closed triple of 2-forms on $X$. There exist triples $\chi \in \cH^2_{\bot}\otimes \RR^3$ and $a \in \Omega^1 \otimes \RR^3$ such that $\theta = \chi + \rd a$ with $\rd ^*a = 0$ and $a_{\bot} = 0$. Moreover, $\chi$ is unique and $a$ is unique up to addition of a triple $b \in \cH^1_{\bot} \otimes \RR^3$. 
\end{lem}
\begin{proof}
By the Hodge theorem for manifolds with boundary there is a unique $\chi \in \cH^2_\bot \otimes \RR^3$ such that $\theta - \chi$ is exact. Write $\theta - \chi = \rd \hat{a}$ for some triple $\hat{a}$ of 1-forms. Now let $f \colon X \to \RR^3$ solve $\Delta f = - \rd^*\hat{a}$, with the Neumann boundary condition $(\rd f)_{\bot} = - \hat{a}_{\bot}$ and write $a = \hat{a} + \rd f$. By choice of $f$, $\rd^* a = 0$ and $a_{\bot} = 0$. 
\end{proof}

\subsubsection{A non-linear operator of mixed order}
\label{mixed.order.section}

We next impose the slice condition $L^*(\theta) = 0$ which, by
Theorem~\ref{main.slice.thm}, is equivalent to dividing out by the
action of diffeomorphisms which are the identity on the boundary (at
least for small~$\theta$).  
\begin{lem}\label{second.order.gauge}
Let $\theta = \chi + \rd a$ where $\chi \in \cH_{\bot}^2 \otimes \RR^3$ and $a \in \Omega^1 \otimes \RR^3$ with $\rd^*a=0$. Then $L^*(\theta) = 0$ if and only if $\Delta(J \cdot a) = 0$.
\end{lem}
\begin{proof}
Note that $L^*(\theta) = (J \cdot \rd^* \theta)^\sharp$ so $L^*(\theta) = 0$
if and only if $J \cdot (\rd^*\rd a) = 0$. Since $\rd^*a = 0$ this is
equivalent to $J\cdot \Delta a = 0$, where $\Delta = \rd^* \rd + \rd
\rd^*$ is the Hodge Laplacian. A hyperk\"ahler metric is Ricci-flat,
so on 1-forms the Hodge Laplacian is equal to the rough
Laplacian. Moreover, $J$ is covariant constant (since each of $J_1,
J_2, J_3$ are) and thus $J$ commutes with the rough Laplacian, and
hence the Hodge Laplacian on 1-forms. The result follows. 
\end{proof}

Thus given $\theta = \chi + \rd a$ as above, the conditions
\begin{equation}\label{e1.29.1.16}
Q(\omega + \chi + \rd a)=0,\; \rd^*a =0,\; \Delta(J\cdot a)=0
\end{equation}
(where $Q$ is defined in equation \eqref{Q}) are equivalent to 
\begin{equation}\label{e2.29.1.16}
Q(\omega +\theta)=0,\; \rd^*a=0,\; L^*(\theta) =0,
\end{equation}
and we know by the slice theorem that these conditions define a
neighbourhood in $\cM$ of $\omega$ if the norm of $\theta$ is
sufficiently small.  We shall combine the three conditions in
\eqref{e1.29.1.16} to define our smooth map $\cQ$, but before doing
so, we must take care of the fact that $\theta$ does not determine $a$
uniquely, even if $\rd^*a=0$. However, Lemma~\ref{coulomb.gauge} shows
us how to fix this problem. Thus we make the following definition.

\begin{dfn}[The gauge-fixed hyperk\"ahler equation]\label{gauge.fixed.hK}
For $(\chi, a)$ as above, define
\begin{equation}
\cQ(\chi,a)
	= 
		\left(
			Q(\omega + \chi + \rd a),
			\rd^*a,
			\Delta (J \cdot a),
			a_{\bot}
		\right).
\label{Q.definition}
\end{equation}
The domain of $\cQ$ is defined to be the open set of
\begin{equation}
\cU^{s+1}
\subset \big((\cH_{\bot}^2\otimes\RR^3) \oplus H^{s+1}(X,T^*X \otimes \RR^3)\big)/(\cH^1_{\bot}\otimes\RR^3)
\end{equation}
satisfying the condition
\begin{equation}
\sum (\omega_i + \chi_i + \rd a_i)^2 >0
\end{equation}
and the degree of regularity $s$ is taken to be $>4$.
\end{dfn}
\begin{rmk}
The map $\cQ$ is well-defined on this domain because $\chi + \rd a$
does not change if a triple of harmonic $1$-forms is added on to $a$. 

Note further that $\cQ$ maps into
\begin{equation}\label{e3.29.1.16}
		H^s(X,S^2_0\RR^3) \oplus
		H^s(X,\RR^3)\oplus 
		H^{s-1}(X,T^*X)\oplus 
		H^{s+1/2}(Y,\Lambda^3T^*Y \otimes \RR^3),
\end{equation}
the last term being essentially the restriction to $Y$ of the triple of normal components of $a$.
\end{rmk}

From \eqref{P}, if 
\begin{equation}\label{e1.6.2.16}
\theta = \chi + \rd a,\;\; a = \sigma + \tau,\;\;
\end{equation}
where
\begin{equation}\label{e1a.6.2.16}
J\cdot \sigma =0
\end{equation}
and $\tau$ is the component of $a$ in the sub-bundle isomorphic to
$S_-S_+$ (recall \eqref{splitting_domain} again), we have
\begin{equation}\label{e2.6.2.16}
Q(\omega + \chi + \rd a)
= Q(\omega) + P(\omega,\chi + \rd a) + Q(\chi+\rd a)
= P(\omega,\chi + \rd a) + Q(\chi + \rd a).
\end{equation}
Thus the linearization of $\cQ$ at $0$ is
\begin{equation} \label{e3.6.2.16}
\rd \cQ_0(\chi,a) = 
(s_0^2(\omega, \chi + \rd \sigma), \rd^*\sigma + \rd^*\tau,
\Delta(J\cdot \tau), a_{\bot})
\end{equation}
using \eqref{e1a.6.2.16}. 
Combining the first two summands to make $S_+\otimes S_+^3$ as in
\eqref{splitting_range}, and using Proposition~\ref{D_is_Dirac}, we obtain
\begin{equation}\label{e4.6.2.16}
\rd \cQ_0(\chi,a) = (\cD \sigma+ \rd^*\tau +
s_0^2(\omega,\chi),\Delta(J\cdot \tau), a_{\bot})
\end{equation}

\subsection{Regularity of $\cQ^{-1}(0)$}
Whilst not strictly speaking an elliptic operator, $\cQ$ is built from elliptic parts. In particular, it enjoys the following regularity property.

\begin{thm}\label{regularity}
Fix $s>4$. There exists $\epsilon >0$ such that if $\cQ(\chi,a)=0$ with $\chi \in \cH_{\bot}^2\otimes \RR^3$, $a \in H^{s+1}(X, T^*X \otimes \RR^3)$ and $\| (\chi,a)\|_{H^{s+1}} < \epsilon$, then in fact $a$ is smooth in the interior of $X$. 
\end{thm}

It follows that a neighbourhood of $[\omega]$ in the moduli space $\cM^s$ of
hyperk\"ahler triples that are smooth in the interior of $X$ and of
regularity $H^{s+\frac{1}{2}}$ on $Y$ is homeomorphic to a
neighbourhood of $\cQ^{-1}(0)$ in $\cU^{s+1}$.

\begin{proof}
We begin with the proof that $a$ is smooth in the interior of
$X$. Note first that $\chi$ is automatically smooth, since it solves
the linear elliptic system $\rd \chi = 0 = \rd^*\chi$.   The component
$\tau$ of $a$ in \eqref{e1a.6.2.16} is smooth, for the map $a \mapsto
J\cdot a$  identifies the sub-bundle of $\Lambda^1\otimes \RR^3$ that
$\tau$ lives in with $\Lambda^1$, and $\Delta(J\cdot \tau)=0$. 

To show that $\sigma$ is smooth, note first from
\eqref{e2.6.2.16} and \eqref{e3.6.2.16} that $\cQ(\chi,a) =0$ is
equivalent to 
\begin{equation}
\cD \sigma = -\rd^*\tau - s^2_0(\omega,\chi) - Q(\chi + \rd(\sigma + \tau)),
\end{equation}
which we write in the schematic form
\begin{equation}
\cD \sigma=
	-q(\rd \sigma, \rd \sigma) - l(\rd \sigma) - r,
\label{s.dirac}
\end{equation}
where $q$ is quadratic and $l$ is linear in $\rd\sigma$, the
coefficients of $q$, $l$ and $r$ all depending real-analytically on
the smooth data $\tau$ and $\chi$.

Equation (\ref{s.dirac}) is a first order fully non-linear  equation
for $\sigma$. To prove regularity one can work directly with the first order equation, but it is more straightforward to use a standard device and take
another derivative to turn (\ref{s.dirac}) into a second order
quasi-linear equation. To do this, we apply the adjoint Dirac operator
$\cD^*$. Schematically, we obtain 
\begin{equation}\label{Lap.s}
\cD^*\cD \sigma
	=
		-\nabla q \cdot \rd \sigma \cdot \rd \sigma 
		-2 q \cdot \nabla (\rd \sigma) \cdot \rd \sigma
		-\nabla l \cdot \rd \sigma
		-l \cdot \nabla (\rd \sigma)
		-\nabla \cdot r,
\end{equation}
where the dots denote various algebraic contractions whose precise form is not important. Write $\mathcal{P} \colon C^\infty(S^3_+\otimes S_-) \to C^\infty(S^3_+\otimes S_-)$ for the second order linear operator
\[
\mathcal{P} (\rho)
	= 
		\cD^*\cD \rho
		+
		2 q \cdot \nabla (\rd \rho) \cdot \rd \sigma
		+
		l \cdot \nabla (\rd \rho).
\]
We have absorbed all the second order behaviour from (\ref{Lap.s}) into $\mathcal{P}$, making it linear by letting $\rd \sigma$ appear in its coefficients. 

The coefficients of $\mathcal{P}$ depend on those of $\cD^*\cD$ and on $\rd \sigma$, $\chi$ and $\tau$. Since $\chi$ and $\tau$ are smooth and $\cD^*\cD$ has smooth coefficients, the coefficients of $\mathcal{P}$ are in the same Holder space as $\rd \sigma$. Since $\sigma \in H^{s+1}$, Sobolev embedding gives that the coefficients are in $C^{k,\alpha}$ for some $k \geq 0$ and $0<\alpha < 1$. (At this stage $k = s - 3$ is the best we can arrange.) 

Next notice that the $C^{0}$ norm of the coefficients of $\mathcal{P}$  depends continuously on $\rd \sigma, \chi, \tau$ (in the $C^0$-topology). Moreover, $\mathcal{P} = \cD^*\cD$ when $\chi = 0 = \tau$ (since $l$ vanishes in this case).  Hence, for $\rd \sigma, \chi, \tau$ sufficiently small in $C^{0}$, and so in particular in $H^{s+1}$, $\mathcal{P}$ is an elliptic operator. 

Now rearranging (\ref{Lap.s}) gives
\begin{equation}
\label{second.order.s}
\mathcal{P} (\sigma) = - \nabla q \cdot \rd \sigma \cdot \rd \sigma - \nabla l \cdot \rd \sigma - \nabla \cdot r.
\end{equation}
Since $\mathcal{P}$ is elliptic with coefficients in $C^{k,\alpha}$
and the right-hand side of \eqref{second.order.s} is in $C^{k,\alpha}$
as well, Schauder estimates apply, giving $\sigma \in
C^{k+2,\alpha}$. It follows in turn that the coefficients of
$\mathcal{P}$ and the right-hand side of \eqref{second.order.s} are
actually in $C^{k+1,\alpha}$ and so $\sigma \in
C^{k+3,\alpha}$. Bootstrapping this argument then gives that $\sigma$
is smooth in the interior of $X$.  
\end{proof}

\begin{rmk}
We stress that while this result gives that all gauge-fixed
hyperk\"ahler perturbations of $\omega$ are smooth in the
interior of $X$, though there is no reason to believe that they will
extend smoothly up to or through the boundary $Y$.
\end{rmk}

\subsection{$\cQ$ is a submersion}

We show next that for any sufficiently large $s$, $\cQ$ is a
submersion.  For this we need to know that $\cD$ is surjective (with
suitable domain and range).  We gather the results we need first,
before proceeding to the proof in \S\ref{s_q_submers}

\subsubsection{On Dirac operators}

Since $\cD$ is an operator of Dirac type, we have
\begin{equation}\label{e21.13.2.16}
\cD^* \cD = \nabla_1^*\nabla_1 + R_1,\;\;
\cD \cD^* = \nabla_2^*\nabla_2 + R_2
\end{equation}
where $\nabla_1$ is the metric connection on $S_-S_+^3$, $\nabla_2$ is
the metric connection on $S_+S_+^3$, $R_1$ is an endomorphism of
$S_-S_+^3$ and $R_2$ is an endomorphism of $S_+S_+^3$.  The
endomorphisms $R_1$ and $R_2$ depend only upon the curvature of the
bundles in question.  Because the only non-vanishing piece of
curvature on a hyperk\"ahler $4$-manifold is the anti-self-dual part of the Weyl
curvature and this is a section of $S_-^4$ (the symmetric fourth power
of $S_-$) it follows that $R_1$ and $R_2$ both vanish identically. 

\begin{prop}\label{p1.13.2.16}
The operator
\begin{equation}
\cD:H^s(X,S_-S_+^3) \to H^{s-1}(X,S_+S_+^3)
\end{equation}
is surjective.
\end{prop}
\begin{proof}
From the formula $\cD\cD^* = \nabla_2^*\nabla_2$, we have that the
spectrum of $\cD\cD^*$ on sections satisfying Dirichlet boundary
conditions is strictly positive.   
Hence, there exists
$$
G :  L^2(X,S_+S_+^3) \longrightarrow H^2(X,S_+S_+^3)
$$
with $\cD\cD^*\circ G = 1$. 
So if $f\in L^2$, $u =\cD^*G f \in
H^1$ and $\cD u = f$.  If we know that $f$ is also in $H^{s-1}$, then
 we still have $\cD u = f$ so elliptic regularity
gives $u \in H^s$.
\end{proof}

\subsubsection{Proof that $\cQ$ is a submersion}
\label{s_q_submers}

We now show that the linearization of $\cQ$ is surjective at every
smooth hyperk\"ahler triple $\omega$.

\begin{prop}
Let $\omega$ be a smooth hyperk\"ahler triple on $X$. Then the
operator $\rd\cQ_0$ in \eqref{e4.6.2.16} is surjective onto
\eqref{e3.29.1.16}.
\end{prop}

\begin{proof}
Let $(\psi,v,b)$ lie in \eqref{e3.29.1.16}. To prove surjectivity of $\rd\cQ_0$, it suffices to find $a=\sigma+\tau$ with
\begin{eqnarray}
\cD \sigma +\rd^*\tau &=& \psi, \label{e91.6.2.16}\\
\Delta (J\cdot \tau) &=& v, \label{e92.6.2.16}\\
a_{\bot} &=& b \label{e93.6.2.16}.
\end{eqnarray}
First, let $\tau'$ solve $\Delta (J\cdot \tau') = v$ with Dirichlet boundary
conditions $\tau'|_Y = 0$. Next we use the surjectivity of $\cD$ on $X$
to solve $\cD \sigma' = \psi - \rd^*\tau'$.  With these choices we have
satisfied \eqref{e91.6.2.16} and \eqref{e92.6.2.16}.   Let $a' =
\sigma' + \tau'$.  We adjust $a'$ so as to satisfy \eqref{e93.6.2.16}
without spoiling \eqref{e91.6.2.16} and \eqref{e92.6.2.16}. 

To do this we consider $a = a'+\rd f$ where $f$ is a triple of
harmonic functions, $\Delta f=0$. We have that  
\[
\rd \cQ_0(\rd f) = \big(s^2_0\left( \rd_+(\rd f)\right) + \rd^* (\rd f), \Delta ( J \cdot \rd f), (\rd f)_{\bot}\big).
\]
Now $\rd_+ (\rd f) = (\rd + * \rd)(\rd f) = 0$ and $\rd^* \rd f = \Delta f = 0$. Moreover, $\Delta (J \cdot \rd f) = J \cdot \Delta \rd f$ (since the metric is hyperk\"ahler) and this also vanishes, since $\Delta \rd f = \rd \Delta f$. The conclusion is that when $f$ is a triple of harmonic functions,
\[
\rd \cQ_0(a' + \rd f)
=
(\psi, v, a'_{\bot} + (\rd f)_{\bot}).
\]
To prove that $\rd \cQ_0$ is surjective, we choose $f$ to be harmonic functions with the Neumann boundary condition $(\rd f)_{\bot} = b - a'_{\bot}$; then $a = a' + \rd f$ solves $\rd \cQ_0 (a) = (\psi, v, b)$.
\end{proof}

We have now shown that $\cQ$ is a submersion. Since we have already
seen that for any smooth triple,  a small neighbourhood of $0$ in
$\cQ^{-1}(0)$ is homeomorphic to a small neighbourhood of $[\omega]$ in
$\cM^s$, we have now proved part (ii) of
Theorem~\ref{smooth_moduli_space}, apart from the identification of
the tangent space in \eqref{e31.11.2.16}.  This will be done in the
next section.

\section{The tangent space to \texorpdfstring{$\cM^s$}{Ms}}
\label{dirac.background.section}

We have now seen that for any smooth hyperk\"ahler triple $\omega$, a
neighbourhood of $[\omega]$ in $\cM^s$ is homeomorphic to a ball
containing the origin in $\ker(\rd \cQ_0)$.  We shall now prove
\eqref{e31.11.2.16}, thereby giving a more satisfactory
interpretation of this tangent space.  The assertion to be proved is
the following.
\begin{claim}
Let $X$ be a compact 4-manifold with boundary $Y$ and $\omega$ 
a hyperk\"ahler triple on $X$.  Then
\begin{equation}\label{claim_eqn}
T_{\omega}\cH^s\cap \ker(L^*) = 
[\cZ^2_-(X)\otimes \RR^3\cap H^s]
+ L(\cW^{s+1})  \subset
H^s(X,\Lambda^2\otimes \RR^3).
\end{equation}
\end{claim}

\begin{rmk}  The space $\cZ^2_-(X)$ of closed anti-self-dual (ASD) $2$-forms consists
  of elements that are smooth in the interior---any element is
  harmonic---but they can be arbitrarily bad at the boundary.  The notation
  $\cap H^s$ means that we consider those closed ASD $2$-forms which
  are in $H^s(X)$, so having boundary values in $H^{s-1/2}(Y)$.
\end{rmk}

This result depends on two facts.  The first is proved exactly as for
the surjectivity of $\cD$.
\begin{lem}\label{l11.13.2.16}
On the hyperk\"ahler manifold $X$ with boundary $Y$, the operator $D =\rd^* + \rd_+$ is surjective.
\end{lem}
\begin{proof} The only thing to check in copying the proof of
  Proposition~\ref{p1.13.2.16} is that $DD^* = \nabla^*\nabla$.  In
  fact, 
\begin{equation}\label{e2.30.1.16}
D^*D =\nabla^*_1\nabla_1\quad\mbox{and}\quad DD^* = \nabla_2^*\nabla_2
\end{equation}
for the same reason that $R_1=0$ and $R_2=0$ in  \eqref{e21.13.2.16}:
the anti-self-dual part of the Weyl curvature cannot act as a non-zero endomorphism of
$S_-S_+$ or $S_+S_+$.
\end{proof}

The second observation we need is contained in the following.
\begin{lem}  \label{l1.12.2.16}
Let $L_+v= \rd_+(\iota_v\omega)$ be the self-dual part of
  the operator $L$. Then
\begin{equation}
L^*L = L^*L_+.
\end{equation}
Furthermore, if $\theta = \theta_+ + \theta_-$ is a triple of closed
$2$-forms decomposed into self-dual and anti-self-dual parts which satisfies $L^*\theta=0$, then we also have
\begin{equation*}
L^*\theta_+ = 0 = L^*\theta_-.
\end{equation*}
\end{lem}
\begin{proof}
Since $L^*\theta = J\cdot \rd^*\theta =-  J \cdot *\rd * \theta$, 
$$
L^*L_+v =  - J\cdot *\rd *(1 + *)\rd (\iota_v\omega) = - J\cdot
*\rd*\rd(\iota_v\omega)
= L^*L v.
$$
For the second part, 
since $\theta$ is closed, we trivially have
\begin{equation}\label{e1.8.2.16}
\rd \theta_+ + \rd \theta_- = 0\quad\mbox{ and so }\quad
J\cdot *\rd \theta_+ +  J \cdot *\rd \theta_- = 0.
\end{equation}
Writing $L^* = -J\cdot (*\rd *)$, we see that $L^*(\theta)=0$ becomes
\begin{equation}\label{e2.8.2.16}
0 = J\cdot *\rd(*\theta_+ + *\theta_-) = J\cdot*\rd \theta_+ - 
J\cdot *\rd \theta_-.
\end{equation}
Combining \eqref{e1.8.2.16} and \eqref{e2.8.2.16}, we obtain
\begin{equation}
J\cdot *\rd \theta_+ = 0 = J\cdot *\rd \theta_-
\end{equation}
from which the result follows because $J\cdot *\rd$ equals $-L^*$ on triples
of self-dual forms and equals $L^*$ on triples of anti-self-dual forms.
\end{proof}

\subsection{Proof of Claim}

In one direction, it is clear that the right-hand side of \eqref{claim_eqn} is contained in the left-hand side.
For the converse, suppose that for any given $\theta \in T_{\omega}\cH^s\cap \ker(L^*)$, we
can solve $L_+ v = \theta_+$ 
for some $v \in C^\infty(X,TX)$.  Write
$$
\theta = (\theta - Lv) + Lv.
$$
Then by construction, $\theta - Lv$ lies in $\Omega^2_-(X)\otimes
\RR^3$.  It is also closed, because $\theta$ is closed by hypothesis
and $L v$ is exact.  So we just have to show
\begin{equation*}
L^*\theta = 0 \Rightarrow L^*L v = 0.
\end{equation*}
However, the second part of Lemma~\ref{l1.12.2.16} applies to give $L^*\theta_+
=0$ and $L^*\theta_-=0$.  Hence if $L_+v = \theta_+$, we have
$$
L^*Lv=L^*L_+ (v) = L^*(\theta_+) = 0,
$$
using the first part of Lemma~\ref{l1.12.2.16} as well.

It remains only to discuss the solvability of $L_+(v) = \theta_+$.

Recall from Proposition~\ref{linearisation.vector.field} that $L_+(v) = \rd_+ J\alpha$, if
$\alpha$ is the $1$-form dual to $v$, and that this map is the
composite of $D = \rd^* + \rd_+$ with the algebraic inclusion 
\begin{equation*}
\Omega^0 \oplus \Omega^2_+ \hookrightarrow \Omega^2_+\otimes
\Omega^2_+.
\end{equation*}
The definition \eqref{e21.12.8.15} of $T_{\omega}\cH^s$ includes the
condition $s_0^2(\omega,\theta)=0$ which says precisely that $\theta_+$
lies in the image of this inclusion.  Since $D$ is
surjective (Lemma~\ref{l11.13.2.16}), it follows that the equation
$L_+(v) = \theta_+$ can be solved for any $\theta \in T_{\omega}\cH^s$.
The proof of the claim is complete, as is the proof of Theorem~\ref{smooth_moduli_space}.

\section{The moduli space \texorpdfstring{$\cM_+$}{M+}}
\label{sect5}

In this section, we make the following assumption.
\begin{assume}\label{a1.31.1.16}
The mean curvature $H$ of $Y$ is everywhere non-negative, and
positive at least one point.  (A definition of $H$ appears in \eqref{e101.8.2.16}.)
\end{assume}

For the avoidance of doubt, our convention is that the mean curvature of
the boundary of the ball in $\RR^4$ is positive.

So far we have only used the surjectivity of Dirac operators on
manifolds with boundary.  We now bring in APS type boundary conditions
for $D$ giving surjectivity of this operator.  We shall refer to the
literature for most of the proofs; we have found 
\cite{BaerBallmann} and \cite{BartnikChrusciel} to be
good references for this material.

\subsection{Geometry near $\p X$}

We first introduce some notation. Let $\rho \colon X \to \RR$ measure the
distance of a point to the boundary $Y$. This function is smooth near
$Y$. Using geodesics which are orthogonal to $Y$ we can identify a
neighbourhood $U$ of $Y$ with the product $[0, \epsilon) \times Y$, with
the function $\rho$ corresponding to projection onto the first
factor.  If $\rho\in [0,\ve)$, $y\in Y$, use parallel transport along
orthogonal geodesics to identify $T_{\rho,y}X$ with $\RR \oplus T_yY$,
and similarly for $1$-forms etc.  The $\RR$ summand here corresponds
to the coefficient of $\nu$, the outward unit vector field tangent to the
orthogonal geodesics.   A consequence of this identification
is that the normal component $\nabla_\nu$ of the metric connection
acts simply as $-\p/\p \rho$ on the $\RR$ and $TY$ components of
vector fields.  We can write any $1$-form $a$ in the form
\begin{equation}\label{e410a}
a = f \rd \rho + b
\end{equation}
where $f$ is a path of functions on $Y$ and $b$ is a path of 1-forms
on $Y$, and 
\begin{equation}
\nabla_\nu a = -\p_\rho f\,\rd \rho - \p_\rho b.
\end{equation}
Similarly the metric takes the form $g = \rd \rho^2 + h(\rho)$ where $h(\rho)$
is a path of Riemannian metrics on $Y$.  Recall that the metric volume
element $\rd \mu_Y$ of the path of metrics $h$ on $TY$ is not closed:
instead we have
\begin{equation}\label{e101.8.2.16}
\rd[\rd \mu_Y] = 
H \rd \mu_X = 
-H \rd \rho\wedge\rd\mu_Y,
\end{equation}
where $H$ is the mean curvature of the family of level sets of
$\rho$.  (We think of $H$ as a path of functions on $Y$.)

Similarly, any self-dual
$2$-form $\theta$ has the form
\begin{equation}\label{e510a}
\theta  = -\rd \rho\wedge c + *_Yc
\end{equation}
in $U$, where $c \in T^*Y$ and $*_Y : T^*Y \to \Lambda^2T^*Y$ is the boundary $*$
operator.  Then mapping $\theta$ to $c = \iota_{\nu}\theta$
identifies $\Lambda^2_+X|U$ with $T^*Y|U$.

\begin{lem}\label{calc.boundary.op}
In the collar neighbourhood $U$ of $Y$, we have that $D=\rd^*+\rd_+$ is given by
\begin{equation}\label{e91.31.1.16}
D:\begin{bmatrix} f \\ b\end{bmatrix}
\longmapsto
\begin{bmatrix}
\nu + H & 0 \\ 0 &  \nu
\end{bmatrix}
\begin{bmatrix}
f \\ b
\end{bmatrix}  + D_Y
\begin{bmatrix}
f \\ b
\end{bmatrix}
\end{equation}
where
\begin{equation}\label{e92.31.1.16}
D_Y = \begin{bmatrix} 0 & \rd^*_Y \\ \rd_Y & *_Y\rd_Y \end{bmatrix}.
\end{equation}
\end{lem}
\begin{proof}
We start by computing $\rd^*a$ for $a$ given in \eqref{e410a}:
\begin{align*}
\rd^*a & =  -*\rd*(f\,\rd\rho + b) \\
  & =  -*\rd(-f \rd \mu_Y - \rd\rho\wedge *_Y b) \\
 & =  -*(\nu(f)\rd\rho\wedge \rd \mu_Y -Hf \rd \mu_X + \rd\rho\wedge \rd_Y*_Y b)
 \\
 & = \nu(f) +Hf  + \rd^*_Y b.
\end{align*}
Similarly,
\begin{equation*}
\rd a = \rd ( f \rd\rho + b) =
-\rd\rho\wedge \rd_Y f -\rd\rho \wedge \nu(b) + \rd_Y b.
\end{equation*}
Hence
\begin{equation*}
(1 +*)\rd a =
-\rd\rho\wedge( \nu(b) + \rd_Y f + *_Y\rd_Y b) +*_Y(\nu(b) + \rd_Y f + *_Y\rd_Y b)
\end{equation*}
which gets identified with 
\begin{equation*}
\nu(b) + \rd_Y f + *_Y\rd_Y b.
\end{equation*}
These computations complete the proof.
\end{proof}

We now turn to the formal adjoint of $D=\rd^*+\rd_+$.
\begin{prop}\label{p31.31.1.16}
The formal adjoint $D^*$ of $D$, with the same identifications, is
given by
\begin{equation*}
D^*
\begin{bmatrix}
f \\ b
\end{bmatrix} \longmapsto
\begin{bmatrix}
-\nu  & 0 \\ 0 &  -\nu - H 
\end{bmatrix}
\begin{bmatrix}
f \\ b
\end{bmatrix}  + D_Y
\begin{bmatrix}
f \\ b
\end{bmatrix}.
\end{equation*}
Moreover, 
\begin{equation}\label{e21.31.1.16}
(D u,v) - (u,D^* v) = \int_{Y} \langle u,v\rangle\rd \mu_Y.
\end{equation}
\end{prop}
\begin{proof}
This follows from our formula \eqref{e101.8.2.16} which shows that $\nu + H$ and
$-\nu$ are formal adjoints to each other.  The second equation also follows from
this formula.
\end{proof}

\subsection{Green's formulae}

By combining \eqref{e21.31.1.16} with the formulae \eqref{e2.30.1.16},
we obtain the following useful result, which will be used to obtain
sharp statements about the injectivity and surjectivity of $D$ with
suitable boundary conditions.

\begin{prop}\label{p1.31.1.16}
Let $D = \rd^* + \rd_+$ on the hyperk\"ahler manifold $X$ with
boundary $Y$. Then   for $u\in \Omega^1(X)$ we have
\begin{equation}\label{e5.11.1.16}
\|D u\|^2 = \|\nabla u\|^2 + \int_{Y} \left(H |\iota_\nu u|^2 +
(u,D_Y u)_Y\right)\rd \mu_Y.
\end{equation}

Furthermore,
\begin{equation}\label{e2.11.1.16}
\|D^* v\|^2 = \|\nabla v\|^2 + \int_{Y}\left( H|\iota_\nu c|^2 -
  (v,D_Y v)_Y\right)
\rd\mu_Y
\end{equation}
for $v \in \Omega^0(X)\oplus \Omega^2_+(X)$, $c$ being the
component of $v$ in $\Omega^2_+(X)$.
\end{prop}
\begin{proof}  For \eqref{e2.11.1.16}, put 
$u = D^*v$ into \eqref{e21.31.1.16}, to get
\begin{equation*} 
(v,DD^* v) - \|D^* v\|^2  = \int_{Y}(v,D^* v)\,\rd\mu_Y.
\end{equation*}
We have an analogous formula for $\nabla^*\nabla$:
\begin{equation*} 
(v,\nabla^*\nabla v) - \|\nabla v\|^2  = -\int_{Y}(v,\nabla_\nu v)\,\rd\mu_Y.
\end{equation*}
Subtracting and recalling that $DD^*=\nabla^*\nabla$ gives
\begin{equation*} 
\|D^*v\|^2 - \|\nabla v\|^2 =-\int_{Y}( v, \nu(v) + D^* v)\,\rd\mu_Y.
\end{equation*}
Now substitute the formula for $D^*$ from
Proposition~\ref{p31.31.1.16} into the right-hand side to obtain
\eqref{e2.11.1.16}.  The formula \eqref{e5.11.1.16} follows in
precisely the same way.
\end{proof}

Another useful result analogous to those in Proposition \ref{p1.31.1.16} relates
the $L^2$-norms of $Du$ and $\wt{D}u$, where 
\begin{equation}\label{e41.31.1.16}
\wt{D} = \rd^* + \rd_-.
\end{equation}

\begin{prop}
Let the notation be as above.  Then, for $u\in \Omega^1(X)$ and $b = \iota^*u$, we have
\begin{equation}\label{e42.31.1.16}
\|\wt{D}u\|^2 - \|D u \|^2 = - 2\int_{Y} (b,*_Y\rd_Y b)\,\rd\mu_Y.
\end{equation}
\end{prop}
\begin{proof} We note that
\begin{equation}\label{e51.31.1.16}
\wt{D}^*\wt{D} = \nabla^*\nabla
\end{equation}
by the same argument that gives \eqref{e2.30.1.16}. 
Computations similar to those in the proofs of Lemma~\ref{calc.boundary.op} 
and Proposition~\ref{p31.31.1.16} give
\begin{equation}\label{e93.31.1.16}
\wt{D} = 
\begin{bmatrix}
\nu + H & 0 \\
0 & \nu \end{bmatrix}
+
\begin{bmatrix}
0 & \rd_Y^* \\
\rd_Y & - *_Y\rd_Y
\end{bmatrix}
\end{equation}
with formal adjoint
\begin{equation*}
\wt{D}^* = 
\begin{bmatrix}
-\nu  & 0 \\
0 & -\nu - H \end{bmatrix}
+
\begin{bmatrix}
0 & \rd_Y^* \\
\rd_Y & - *_Y\rd_Y
\end{bmatrix}
\end{equation*}
in the collar neighbourhood $U$ of $Y$.  Arguing now as in the
proof of Proposition~\ref{p1.31.1.16}, we obtain the formulae
\begin{equation*} 
(u,D^*D u) - \|D u\|^2 = - \int_{Y} (u,D u)\,\rd\mu_Y
\end{equation*}
and
\begin{equation*} 
(u,\wt{D}^*\wt{D} u) - \|\wt{D} u\|^2 = - \int_{Y} (u,\wt{D} u)\,\rd\mu_Y.
\end{equation*}
The first term on the left-hand side of each of these two equations is
$(u,\nabla^*\nabla u)$, so subtracting we obtain
\begin{equation}
\|\wt{D}u\|^2 - \|Du\|^2 = \int_{Y} (u, (\wt{D} - D)u)\,\rd\mu_Y.
\end{equation}
The result now follows from our formulae for $\wt{D}$ and $D$,
\eqref{e91.31.1.16} and \eqref{e93.31.1.16}.
\end{proof}

\subsection{The kernel of $D$ in terms of boundary data}

We shall now combine the formulae just obtained with standard Fredholm
results for operators of Dirac type on a manifold with boundary to
parameterize the null space of $D$ in terms of boundary data.

The operator $D_Y$ is (formally) self-adjoint and of first order, so
it has a discrete real spectrum which is unbounded above and below,
with no (finite) accumulation points.  Denote by $H_{\lambda}$ the
eigenspace of $D_Y$ corresponding to the eigenvalue $\lambda$.  Fix a
real number $s>1/2$. 
\begin{dfn}
Denote by $H^{s-1/2}_+(Y)$ the completion in the Sobolev
$(s-1/2)$-norm of $\oplus_{\lambda>0} H_\lambda$. Similarly, denote by
$H^{s-1/2}_-(Y)$ the completion in the $(s-1/2)$-norm of
$\oplus_{\lambda <0} H_\lambda$.
\end{dfn}
\begin{rmk} We shall refer to the elements of $H^{s-1/2}_+(Y)$ as
  positive frequency boundary data, and similarly to the elements of
  $H^{s-1/2}_-(Y)$ as negative frequency boundary data.
\end{rmk}

Then we have
\begin{equation}\label{e1.26.1.16}
H^{s-1/2}(Y) = H^{s-1/2}_-(Y) \oplus H_0(Y) \oplus H^{s-1/2}_+(Y),
\end{equation}
with $H_0(Y)$ being the (finite-dimensional) kernel of $D_Y$.
Similarly define:
\begin{align*}
H^s_+(X) &= \{u \in H^s(X,\Lambda^1): u|Y \in H_+^{s-1/2}(Y)\};\\
H^s_-(X) &= \{u \in H^s(X,\Lambda^1): u|Y \in H_-^{s-1/2}(Y)\};\\
H^s_0(X) &= \{u \in H^s(X,\Lambda^1): u|Y \in H_0^{s-1/2}(Y)\}.
\end{align*}
The basic results we need are as follows.
\begin{thm}
Let $X$ be a hyperk\"ahler manifold with smooth boundary and mean curvature
$H\geq 0$, and strictly positive at at least one point.   Then for
$s>1/2$, the operator
\begin{equation}\label{e1.11.1.16}
D=\rd^*+\rd_+ : H^s_{\geq 0}(X,\Lambda^1) \to 
H^{s-1}(X,\RR \oplus \Lambda^2_+)
\end{equation}
is surjective, with finite-dimensional kernel isomorphic to $H^1(X)$.

Further, there is a Poisson operator
\begin{equation*}
\cP: H^{s-1/2}_{-}(Y) \to \Ker(D)\cap H^s(X,T^*X),
\end{equation*}
i.e.~the projection to $H^{s-1/2}_-(Y)$ of the restriction $\cP(f)|Y$
is equal to $f$.
\label{t1.26.1.16}\end{thm}
\begin{rmk} Here we have written $H_{\geq 0}$ for the direct sum of
  $H_+$ and $H_0$.
\end{rmk}
\begin{proof}
Without any restriction on the mean curvature, that \eqref{e1.11.1.16}
is Fredholm is standard in the theory of Dirac operators on manifolds
with boundary \cite{BaerBallmann,BartnikChrusciel}.  This theory also identifies the cokernel
of \eqref{e1.11.1.16} with the null-space of the adjoint operator $D^*$
with domain $H^{s}_-(X)$.

Consider \eqref{e2.11.1.16} applied to $v$ with
\begin{equation}\label{e102.8.2.16}
D^* v=0, \quad v \in H^s_-(X).
\end{equation}
The first term on the right-hand side of \eqref{e2.11.1.16}  is manifestly $\geq 0$, the second term is $\geq 0$ by
Assumption~\ref{a1.31.1.16} and the third is strictly positive if
$0\neq v|Y \in H^{s-1/2}_-(Y)$ by \eqref{e102.8.2.16}.  But the left-hand side of \eqref{e2.11.1.16} is $0$ by \eqref{e102.8.2.16}, which means that
$v|Y=0$ and $\nabla v=0$.  Hence $v$ is identically zero and $D^*$ is
injective on $H^s_-(X)$. 

To identify the kernel of $D$ we use the formula
\eqref{e5.11.1.16}.  We see that if $D u=0$ then $\nabla u =0$ and
$D_Y(u|Y)=0$. Looking at the formula for $D_Y$, it follows that if we write 
$u=f\rd\rho+b$ on $Y$ then
$$
\rd_Y f = 0, \quad \rd_Y b = 0 = \rd^*_Yb.
$$
Hence $f$ is constant and $\int_Y Hf^2 = 0$ implies $f=0$ if $H\geq 0$
and is positive at a point.

Now the standard Weitzenb\"ock formula for $1$-forms on a Ricci-flat
$4$-manifold shows that every harmonic 1-form $a$ with $\iota_\nu(a)=0$
is parallel.  So the null space of \eqref{e1.11.1.16} is isomorphic to this space of forms,
which is in turn identifiable with $H^1(X)$ by Hodge theory.

The construction of the Poisson operator is standard, but we recall
the details. For any given $s$, we can define a bounded extension
operator
$$
E : H^{s-1/2}_-(Y) \longrightarrow H^s(X,\Lambda^1)
$$
so that $Ef|Y =
f$.   Let $G : H^{s-1}(X,\RR\oplus\Lambda^2_+) \to H^{s}_{\geq 0}(X,\Lambda^1)$ be a right-inverse of \eqref{e1.11.1.16}.  
Set
$$
\cP f = Ef - GD(Ef).
$$
By definition $\cP$ maps into $\Ker(D)\cap H^s$.  Since $(G\sigma|Y)_-
=0$ for any $\sigma$ in $H^{s-1}(X,\RR\oplus\Lambda^2_+)$,  it follows that $(\cP f|Y)_- = (Ef|Y)_- = f$ as required.
\end{proof}

\subsection{The kernel of $D$}

 We now wish to give a precise description of $\ker(D)$
  in terms of boundary data.  Recall the decomposition
  \eqref{e1.26.1.16} of boundary data
\begin{equation}\label{e1.12.10.16}
H^{s-1/2}(Y) = H^{s-1/2}_-(Y) \oplus H_0(Y) \oplus H^{s-1/2}_+(Y),
\end{equation}
in terms of the spectrum of $D_Y$, and that the coefficient bundle
here is $T^*X|Y = \RR \oplus T^*Y$. 

The finite-dimensional space $H_0(Y)$ consists of pairs $(f,b)$ where
$f$ is a constant function and $b$ is a harmonic $1$-form on $Y$. Split
\begin{equation}
H_0(Y) = H_{0,-}(Y) \oplus H_{0,+}(Y)
\end{equation}
where
\begin{equation*}
H_{0,-}(Y) = \im(H^1(X) \rightarrow H^1(Y))
\end{equation*}
and $H_{0,+}(Y)$ is the orthogonal complement of this space in
$H_0(Y)$.

\begin{lem}\label{l1.11.10.16}
Suppose that $u\in \ker(D)\cap H^s$ and $u|Y \in H_0(Y) \oplus
H^{s-1/2}_+(Y)$.    Then if $H\geq 0$ and is strictly positive at at least one point, it follows that the component
$u_0$ of $u|Y$ in $H_0(Y)$ must lie in $H_{0,-}(Y)$ and the
positive-frequency part $u_+$ of $u|Y$ is zero.
\end{lem}
\begin{proof}
For $u$ as given, we have, from \eqref{e5.11.1.16},
\begin{equation}\label{e2.12.10.16}
0 = \|\nabla u\|^2 + \int_Y \left(H |\iota_\nu u|^2 +
(u,D_Y u)_Y\right)\rd \mu_Y.
\end{equation}
and all terms on the RHS are separately $\geq 0$. Hence they are all
zero. It follows in particular that $\int_Y (u,D_Yu)_Y\rd \mu_Y=0$, so
$u_+=0$.  Thus, $u|Y = u_0 = (f,b)$ where $f$ is constant and
$b$ is harmonic. Since $H\geq 0$ with strict inequality at some point, 
$f=\iota_{\nu} u=0$. 
 Moreover, $\nabla u=0$ in $X$ which, since $X$ is Ricci-flat, is equivalent to $\rd u =0 = 
\rd^* u$.  Thus $u \in \cH^1_{\bot}(X)$, given in \eqref{Neumann.eq}, which is isomorphic to
$H^1(X)$, and $b$ is its restriction to the boundary.  Since $b=0$
implies $u=0$ (since $\nabla u=0$ in $X$), we have that $u$ is uniquely
 determined by its boundary value $b$, which defines a unique element in 
 $H_{0,-}(Y)$.  
\end{proof}

Combining Lemma \ref{l1.11.10.16} with Theorem~\ref{t1.26.1.16}, we obtain the following.
\begin{prop}\label{kerd}
Under the positive mean curvature assumption \ref{a1.31.1.16},
we have a natural isomorphism
\begin{equation}\label{e11.11.10.16}
\ker(D)\cap H^s \cong H_-^{s-1/2}(Y) \oplus H_{0,-}(Y).
\end{equation}
\end{prop}
\begin{proof}
The map is given by restriction to the boundary followed by projection
onto $H_-^{s-1/2}(Y) \oplus H_{0,-}(Y)$.

Given $v = (v_-,v_0)\in H_-^{s-1/2}(Y) \oplus H_{0,-}(Y)$, by
definition there exists $u_0$ with $Du_0=0$ and $u_0|Y =
v_0$. Then
\begin{equation}
\cP v_- + u_0\in \ker D
\end{equation}
and the projection to $H_-^{s-1/2}(Y) \oplus H_{0,-}(Y)$ of this element is
$(v_-,v_0)$.  Hence \eqref{e11.11.10.16} is surjective.

Conversely, suppose $u \in \ker(D)$ has $u|Y \in 
H_0(Y)\oplus H_+^{s-1/2}(Y)$. By Lemma \ref{l1.11.10.16}, $u|Y \in H_{0,-}(Y)$
and this proves that \eqref{e11.11.10.16} is also injective.
\end{proof}

\subsubsection{More Hodge theory}

On our compact manifold $X$ with boundary inclusion $\iota: Y \to X$, the intersection pairing is
well-defined on the space
\begin{equation}\label{e5.13.9.15}
\ker(H^2(X) \to H^2(Y)) = \im(H^2(X,Y) \to H^2(X))
\end{equation}
by the usual formula
\begin{equation}\label{e6.13.9.15}
[\alpha]\cup[\beta]=\int_X \alpha\wedge\beta,
\end{equation}
where we need $\iota^*(\alpha) = \iota^*(\beta)=0$ for this to be
well-defined in cohomology.   
Thus we may choose a decomposition
\begin{equation}\label{e1.19.3.16}
\ker(H^2(X)\to  H^2(Y)) = H^2_+(X) \oplus H^2_-(X)
\end{equation}
such that \eqref{e6.13.9.15} is positive-definite on $H^2_+(X)$,
negative-definite on $H^2_-(X)$.  By choosing a  
complement $H^2_0(X)$ of \eqref{e5.13.9.15} in $H^2(X)$, 
we complete \eqref{e1.19.3.16} to a decomposition
\begin{equation}\label{e8.13.9.15}
H^2(X) = H^2_+(X) \oplus H^2_-(X) \oplus H^2_0(X).
\end{equation}
and the dimensions of these spaces depend only on the topology of the
pair $(X,Y)$. 

By the Hodge theory in Theorem~\ref{hodge.boundary}, 
\begin{equation}\label{Neumann.eq.2}
H^2(X) \cong \cH^2_{\bot}=\{\alpha\in\Omega^2(X):\rd\alpha=\rd^*\alpha=0,\,\iota^*(*\alpha)=0\}
\end{equation} and
so we have an isomorphism
\begin{equation}\label{e2.19.3.16}
\ker(H^2(X) \to H^2(Y)) \cong
\{\alpha \in \cH^2_{\bot} :  [\iota^*\alpha] = 0\in H^2(Y)\}.
\end{equation}
Notice that the projections from $\Lambda^2$ to $\Lambda^2_{\pm}$ given by
\begin{equation}\label{e3.19.3.16}
P_{\pm}(\alpha) = \frac{1}{2}(\alpha \pm * \alpha)
\end{equation}
map closed and coclosed 2-forms to $\cZ^2_{\pm}(X)$ and if $\iota^*(*\alpha)=0$ then $\iota^*(2P_{\pm}(\alpha))=\iota^*(\alpha)$. Thus, if we define 
finite-dimensional spaces
\begin{align}
\cH^2_{\pm}(X)&=\{P_{\pm}(\alpha):\alpha\in\cH^2_\bot,\quad [\iota^*\alpha]=0\in H^2(Y)\}\subset\cZ^2_{\pm}(X),\label{H2pm.eq}\\
\cH^2_0(X) &=\{\alpha\in\cH^2_{\bot}: [*\alpha]=0\in H^2(X)\},\\
\cH^2_{0,\pm}(X)&=P_{\pm}\cH^2_0(X)\subset \cZ^2_{\pm}(X),\label{H20pm.eq}
\end{align}
where $\cZ^2_{\pm}(X)$ are the closed self-dual/anti-self-dual 2-forms on $X$, then we have the following.

\begin{thm}\label{cohom.thm}  In the notation above, 
$\cH^2_{\pm}(X)\cong H^2_{\pm}(X)$ and $\cH^2_{0,\pm}(X)\cong H^2_0(X)$.
\end{thm}

\begin{proof}
Given  the isomorphism \eqref{e2.19.3.16} and the fact that $P_++P_-=\id$, we see that 
$$\cH^2_+\oplus\cH^2_-\cong \ker(H^2(X)\to H^2(Y)).$$
Moreover, the intersection-form is
positive-definite on $\cH^2_+$ and negative-definite on $\cH^2_-$.   
The decomposition \eqref{e1.19.3.16} then implies that 
$\cH^2_{\pm}\cong H^2_{\pm}(X)$. 

Now consider $H^2_0(X)$, which is
isomorphic to the cokernel of the map
\begin{equation}\label{e5.19.3.16}
H^2(X,Y) \to H^2(X).
\end{equation}
Since the $*$-operator interchanges the spaces in \eqref{e5.19.3.16}, it also interchanges
the kernel and cokernel of this map. In particular, \eqref{e1.19.3.16}
is complemented in $H^2(X)$ by the classes represented by $\alpha \in \cH^2_{\bot}(X)$ such
that $*\alpha$ is in the kernel of \eqref{e5.19.3.16}; that is, by
\begin{equation}\label{H20.eq}
\{[\alpha]\in H^2(X):\alpha\in\cH^2_0(X)\}.
\end{equation}  In fact, \eqref{H20.eq} is the $L^2$ orthogonal complement of
\eqref{e1.19.3.16}.  (It is easy to see that these two spaces are
orthogonal inside $H^2(X)$, and the argument just given
shows that they span $H^2(X)$.)  Thus we may set $H^2_0(X)$ equal to \eqref{H20.eq} so that \eqref{e8.13.9.15} holds.

For $\alpha\in \cH^2_0(X)$,  we see that $[2P_{\pm}(\alpha)]=[\alpha]\in H^2(X)$.  
It follows that $\cH^2_{0,\pm}(X)=P_{\pm}\cH^2_0(X)$ are isomorphic to $H^2_0(X)$ as claimed.
\end{proof}

\begin{rmk} 
The fact that we can choose a complement $H^2_0(X)$ of \eqref{e1.19.3.16} in $H^2(X)$
 which can be represented equally well
by self-dual or anti-self-dual forms shows clearly that the cup
product is not well-defined on this space!
\end{rmk}

The next result gives a `standard form' for any element of $\cZ^2_-(X)$.

\begin{prop}\label{z2-.splitting.prop} We have the following direct sum decomposition:
\begin{equation}\label{z2-.splitting.eq}
\cZ^2_-(X)= \cH^2_{0,-}(X)\oplus\cH^2_-(X)\oplus\{\rd a\in
\rd\Omega^1(X)\,:\,(\rd^*+\rd_+)a=0\}.
\end{equation}
Moreover, with respect to the decomposition of $1$-forms in a collar neighbourhood
$$
a = f\,\rd \rho + b
$$
(cf.\ \eqref{e410a}) we may assume $f|Y=0$.
\label{p1.13.9.15}\end{prop}

\begin{proof}  
It is clear that the right-hand-side of \eqref{z2-.splitting.eq} is contained in $\cZ^2_-(X)$
since if $\rd_+a=0$ then $\rd a$ is anti-self-dual and exact.  Let
$\alpha \in \cZ^2_-(X)$. By Theorem \ref{cohom.thm}, the corresponding
cohomology class $[\alpha]$ has components only in $H^2_-(X)\oplus
H^2_0(X)$, and these have unique representatives $(\alpha_-,\alpha_0)
\in\cH^2_-(X)\oplus \cH^2_{0,-}(X)$. Then $\alpha - \alpha_--\alpha_0$ is
exact, so we may write
\begin{equation*} 
\rd a' = \alpha - \alpha_- - \alpha_0,
\end{equation*}
and automatically
\begin{equation*} 
\rd_+a'=0.
\end{equation*}
Suppose further that
$$
a' = f'\,\rd \rho + b'\mbox{ near }Y.
$$
We have not yet arranged $f'|Y=0$ or  $\rd^*a'=0$.  For this, 
define $a=a' + \rd u$, so $\rd a = \rd a'$,
$$
\rd^*a = \rd^*\rd u + \rd^*a'
$$
and if $a = f\,\rd \rho  + b$ then
$$
f = f' + \p_\rho u \mbox{ on }Y.
$$
Solving Poisson's equation $\rd^*\rd u = - \rd^*a'$ with the Neumann
condition $\p_\rho u|Y = -f'|Y$ yields $a$ satisfying $\rd^*a=0$ and 
$f|Y=0$ as required.
\end{proof}

Let 
$$
K^s = \{a \in \ker(D)\cap H^s(X,\Lambda^1): a_{\bot} =0\}.
$$
Proposition \ref{z2-.splitting.prop} shows that $\rd K^s$ is isomorphic to the
 space of exact ASD $2$-forms.  The next result shows that, up to $H^1(X)$, $\rd$ gives an
isomorphism of $K^s$ onto $\rd K^s$.
\begin{prop}  With the above definitions, 
the following sequence is exact:
\begin{equation}
0 \to H^1(X) \to K^s \to \rd K^s \to 0.
\end{equation}
\end{prop}
\begin{proof}
Proposition~\ref{p1.13.9.15} shows that the sequence is exact at $\rd
K^s$.  It is also clear that it is exact at $H^1(X)$ and that it is a
complex. It remains to show that the kernel of $\rd$ is precisely 
$H^1(X)$, identified as $\cH^1_{\bot}(X)$, the harmonic $1$-forms $a$ with
$a_{\bot}=0$. Suppose $\rd a=0$, with $a\in K^s$. Since $\rd^*a=0$ and $a_{\bot}=0$
as part of the definition of $K^s$, 
 $a\in\cH^1_{\bot}(X)$ as required.
\end{proof}

\begin{rmk}  Recall that $H_{\lambda}$ is the $\lambda$-eigenspace of $D_Y$.  
For real $\lambda$, put
\begin{equation}\label{Glambda.eq}
G_\lambda = \{ u \in H_{\lambda} : \rd^* u = 0\}.
\end{equation}
Clearly $G_\lambda$ is finite-dimensional for every $\lambda$ and the
set of $\lambda$ with $G_{\lambda}\neq 0$ is discrete.  It can also be
shown that the set of $\lambda$ with $G_\lambda\neq 0$ is 
unbounded above and below, just as for the $H_\lambda$.  
Denote by $G^{s-1/2}_-(Y)$ the completion of the direct sum
$\oplus_{\lambda<0} G_\lambda$.   Then $K^s$ is isomorphic to
$H_{0,-}(Y)\oplus G^{s-1/2}_-(Y)$ (and is infinite-dimensional), and $\rd
K^s$ is isomorphic to $G^{s-1/2}_-(Y)$.  This follows at once from \eqref{e42.31.1.16}.
\end{rmk}

We may now prove Theorem~\ref{asd_boundary_values}, for which
we need a definition of $\cW_+$.  Recall the splitting
\begin{equation*}
H^{s-1/2}(Y) = H^{s-1/2}_-(Y)\oplus H_{0,-}(Y) \oplus H_{0,+}(Y) \oplus
H^{s-1/2}_+(Y)
\end{equation*}
where the suppressed bundle is $T^*X|Y = \RR \oplus T^*Y$.  The space
$\cW^s$ is by definition
\begin{equation*}
\cW^s = \ker(L^*L)\cap H^{s}
\end{equation*}
and restriction to the boundary gives an isomorphism
\begin{equation*}
\cW^s\cong H^{s-1/2}(Y)
\end{equation*}
(where we identify vector fields with $1$-forms using the
metric). Define
\begin{equation*}
\cW_+^s = \{w \in \cW^s : w|Y \in H_{0,+}(Y) \oplus H^{s-1/2}_+(Y)\}.
\end{equation*}

We shall prove the following sharpened version of Theorem~\ref{asd_boundary_values}.
\begin{thm}\label{t1.18.2.16}
Let $\omega$ be a smooth hyperk\"ahler triple on $X$, inducing
positive mean curvature on the boundary $Y$.  Then the gauge-fixed tangent space
\begin{equation}\label{e21.18.2.16}
T_{[\omega]}\cM^s_+ = T_{\omega}\cH^s \cap \ker(L^*)
\end{equation}
is isomorphic to the direct sum
\begin{equation}\label{e22.18.2.16}
[\cZ^2_-(X)\otimes \RR^3 \cap H^s]
\oplus L(\cW_+^{s+1}).
\end{equation}
Moreover, the summands are naturally isomorphic to the spaces of boundary values
\begin{equation}\label{e22.11.2.16}
\cZ^2_-(X)\otimes \RR^3 \cap H^s  \cong \cH^2_{0,-}(X)\oplus \cH^2_-(X) \oplus G^{s+1/2}_-(Y)
\end{equation}
and
\begin{equation}\label{e21a.18.2.16}
L(\cW_+^{s+1})  \cong H_{0,+}(Y)\oplus H^{s+1/2}_+(Y).
\end{equation}
\end{thm}
\begin{proof}
Note first of all that $L$ is injective on $\cW_+^{s+1}$. Indeed, if
$w\in \cW_+^{s+1}$ and $Lw=0$, then in particular $L_+w=0$.  But
Proposition~\ref{linearisation.vector.field} shows that $L_+w$ can be
identified with $Dw$. By Lemma~\ref{l1.11.10.16}
$Dw=0$ and $w\in \cW_+^{s+1}$ implies that $w=0$.

The same argument shows
\begin{equation*}
[\cZ^2_-(X)\otimes \RR^3]\cap L(\cW^{s+1}_+) = 0.
\end{equation*}
Indeed, if $w\in \cW^{s+1}_+$ is such that $Lw$ lies in the intersection, then $L_+w=0$, so $w=0$ as before. 

Since \eqref{e22.11.2.16} follows from our earlier discussion and the isomorphism \eqref{e21a.18.2.16} follows from the injectivity of $L$ on
$\cW_+^{s+1}$, it remains only to prove that the direct sum
\eqref{e22.18.2.16} is equal to the tangent space as given in
\eqref{claim_eqn}:
$$
[\cZ^2_-(X)\otimes \RR^3\cap H^s]
+ L(\cW^{s+1})  \subset
H^s(X,\Lambda^2\otimes \RR^3).
$$
For this, let $w \in \cW^{s+1}$ and 
\begin{equation}\label{e31.18.2.16}
Lw  = L_+w + L_-w
\end{equation}
be the self-dual/anti-self-dual decomposition of the triple $Lw$.
Since $L_-w=\rd_-(\iota_w\omega)\in \cZ^2_-(X)\otimes \RR^3\cap H^s$, we just need to show that we can find $w' \in \cW^{s+1}_+$ with
$$
L_+w = L_+w'.
$$
Let
the boundary value of $w$ be written $w_- + w_+$ where
\begin{equation}\label{e21.11.10.16}
w_- \in H^{s-1/2}_-(Y)\oplus H_{0,-}(Y),\quad
w_+ \in H^{s-1/2}_+(Y)\oplus H_{0,+}(Y).
\end{equation}
Using Proposition~\ref{kerd}, we find $u$ with $D u =0$ and $u|Y = w_-
+ u_+$, where $u_+\in H^{s-1/2}_+(Y)\oplus H_{0,+}(Y)$. Recalling
again that $D = L_+$ and that $L^*L=L^*L_+$ by Lemma \ref{l1.12.2.16}, if we define
$$
w' = w - u,
$$
then we have
$$
L_+w' = L_+w,\quad  L^*Lw' = L^*L_+w' = L^*L_+w = 0,
$$
and $w'|Y=w_+-u_+$ is positive frequency.  Hence $w' \in \cW^{s+1}_+$ with $L_+w=L_+w'$ as required.
%
%
\end{proof}

\subsection{Proof of Theorem~\ref{global_moduli_space}}
\label{sglobal_mod}

We now show that the moduli space $\cM_+$ of smooth (up to the boundary)
hyperk\"ahler triples inducing positive mean curvature on the boundary
is a manifold.

First we note that $\cM_+$ is well-defined: every smooth hyperk\"ahler
triple $\omega$ (or rather its $\cG_0^{s+1}$-equivalence class) has a
neighbourhood in $\cM_+^s$ homeomorphic to a ball in 
$$
[\cZ^{2}_-(X)\otimes \RR^3 \cap H^s]\oplus L(\cW^{s+1}_+).
$$
The elements of this ball are smooth in the interior and of finite
regularity at the boundary.  However, the parameterization in terms of
boundary values shows that there is a non-zero subspace of smooth
elements of this space: simply choose boundary values in $H^{s}$ on $Y$
for every $s$ (and also satisfying the relevant frequency conditions).

The issue is that the gauged-fixed tangent spaces 
\begin{equation}\label{cT.eq}
T_{[\omega]}\cM_+\cong \cT=[\cZ^2_-(X)\otimes \RR^3] \oplus L(\cW_+)
\end{equation}
depend on $\omega$: the notion of anti-self-duality depends on the
metric, as does $L$, and the operator $D_Y$, which defines
the frequency decomposition that defines $\cW_+$.

Although these spaces move, the claim is that they are all naturally
isomorphic on the path components of $\cM_+$.

\begin{prop}\label{tangent.proj.prop}
Let $\omega_0$ and $\omega_1$ be two smooth hyperk\"ahler triples in
the same path component of $\cM_+$.  Let $\cT_0$ and $\cT_1$ be the
corresponding gauge-fixed tangent spaces as given by \eqref{cT.eq}.  Then the restriction to $\cT_0$ of the
$L^2$-orthogonal projection on $\cT_1$ is an isomorphism.
\end{prop}
\begin{proof}
Denote by $\cT_i^{\perp}$ the $L^2$-orthogonal complement of $\cT_i$ in
$L^2(X,\Lambda^2\otimes \RR^3)$.   Note first the standard fact that
the restricted $L^2$-orthogonal projection maps are isomorphisms if and only if 
\begin{equation}\label{e35.18.2.16}
\cT_0\cap \cT_1^{\perp} = 0 = \cT_1\cap \cT_0^{\perp}.
\end{equation}
To see this, let $\pi : \cT_0 \to \cT_1$ be the restricted projection map. Then $\pi$
is injective if and only if $\cT_0\cap \cT_1^{\perp}=0$.  If $\pi$ is not
surjective, there is $\xi\in \cT_1$, orthogonal to the image $\pi(\cT_0)$
of $\cT_0$ in $\cT_1$. If $\eta \in \cT_0$ and we write $\eta = \eta_1 +
\eta_1^{\perp}\in \cT_1\oplus\cT_1^{\perp}$, then
\begin{equation*} 
\langle \xi,\eta\rangle = \langle \xi,\eta_1\rangle =0
\end{equation*}
because $\pi(\eta) = \eta_1$.  This is true for all $\eta \in \cT_0$ so $\xi \in \cT_1\cap
\cT_0^{\perp}$.  So the assumption \eqref{e35.18.2.16} implies that
$\xi=0$, and $\pi$ is surjective.

It therefore suffices to prove \eqref{e35.18.2.16}.  By hypothesis, there is
a path of hyperk\"ahler triples $\omega(t)$, $0\leq t\leq 1$,
connecting $\omega_0$ to $\omega_1$ in $\cM_+$, and a corresponding
continuous path $\cT_t$ of gauge-fixed tangent spaces.  If one of
\eqref{e35.18.2.16} fails, then we may suppose by symmetry that $\cT_1\cap
\cT_0^\perp\neq 0$. 

We shall use the boundary value description:
\begin{align*}
\cZ^2_-(X)\otimes \RR^3  \cong  \cH^2_{0,-}(X)\oplus \cH^2_-(X) \oplus G_-(Y)
\quad\text{and}\quad
L(\cW_+)  \cong H_{0,+}(Y)\oplus H_+(Y).
\end{align*} 
Note that $G_\lambda$ given in 
\eqref{Glambda.eq} can also be characterized as the subspace of 
$H_\lambda$, the $\lambda$-eigenspace of $D_Y$, with the function
component zero. So we have a decomposition 
$$
\Omega^1(Y) = G_-(t)\oplus G_0(t) \oplus G_+(t)
$$
for all $t$.  Notice that 
$G_0(t)$ consists of the harmonic 1-forms on $Y$, so $G_0(t)\cong H^1(Y)$.   Let $F_+(t)$ denote the space $H_{0,+}(Y)\oplus H_+(Y)$ as
defined by $\omega_t$ and $F_-(t)$ be its orthogonal complement.   Thus
the boundary values of $\cW_+(t)$ lie in $F_+(t)$.   Moreover, recall 
that $H_{0,+}(Y)\cong H^1(Y)/\im(H^1(X)\to H^1(Y))$, so $H_{0,+}(Y)$ has topologically determined dimension.

Suppose $\cT_1\cap \cT_0^\perp =0$ fails.  Then we have
$$
G_-(1) \cap [G_0(0)\oplus G_+(0)]\neq 0\mbox{ or } F_+(1) \cap F_-(0) \neq 0.
$$
Suppose the first possibility occurs.  Then for some $t$, $G_-(t)$ contains
an element of $G_0(0)$. However, as we observed, $G_0(t)$ is of fixed
 dimension equal to $\dim H^1(Y)$, giving a contradiction.  The second
  possibility is ruled out for a similar reason, since $H_{0,+}(Y)$ has a 
  fixed dimension.
\end{proof}

We now prove Theorem~\ref{global_moduli_space}
which we restate for
convenience. 

\begin{thm*}
  The moduli space $\cM_+$ of hyperk\"ahler triples on $X$ inducing positive mean curvature on the boundary $Y$, modulo the action of $\cG_0$, is a 
  Fr\'echet manifold with
  $$T_{[\omega]}\cM_+=[\cZ^2_-(X)\otimes\RR^3]\oplus L(\cW_+).$$
  It should be noted that the spaces on the right-hand side depend on $\omega$.
\end{thm*}  

\begin{proof}  We have seen that on each connected component all
  tangent spaces are canonically identifiable with each other.  It
  follows from this that the transition maps between different
  coordinate patches are smooth as follows.

On any component of an overlap between two charts, which are
necessarily determined by $[\omega_0]$ and $[\omega_1]$ which are
path-connected, $\cM_+$ can be written as a smooth graph over
the tangent spaces $\cT_{0}$ and $\cT_{1}$.   
Since $\cT_{1}$ is a graph over $\cT_{0}$ by Proposition \ref{tangent.proj.prop}, the transition map on the component will be a composition 
of projections from and to smooth graphs over open sets in Fr\'echet
spaces, and thus is smooth.
\end{proof}

\section[$\SU(2)$-invariant examples]{SU(2)-invariant examples}
\label{examples}

Complete $\SU(2)$-invariant hyperk\"ahler metrics in 4 dimensions have
been well understood for many years \cite{MR934202, GibbonsPope}.  We
give a brief description of the classification from our present point
of view as a further illustration of the formalism of triples and to
justify explicitly the claim of the Introduction that a given metric
on a $3$-manifold can arise by restriction of two non-isometric
hyperk\"ahler metrics.

The $\SU(2)$-invariant hyperk\"ahler metrics fall into two classes, according as the corresponding
hyperk\"ahler triple is fixed or rotated under the $\SU(2)$-action. 
 In both cases one seeks hyperk\"ahler triples of the form $\omega=\rd t\wedge\eta_t+*_t\eta_t$ where $\eta_t$ is a family of left-invariant coclosed coframes on $\SU(2)$ (or quotients thereof) and $*_t$ is the induced Hodge star on each hypersurface in the 4-manifold given by fixing $t$.  We briefly review the analysis of these
gravitational instantons.

For the case where the triple is fixed one chooses the standard left-invariant coframing $\eta=(\eta_1,\eta_2,\eta_3)$ of $\SU(2)$ such that $\rd \eta_i=\epsilon_{ijk}\eta_j\wedge\eta_k$ and considers 
$\omega=(\omega_1,\omega_2,\omega_3)$ where
\begin{equation}\label{e1.13.8.15}
\omega_1=f_1\rd t\wedge \eta_1+f_2f_3 \eta_2\wedge \eta_3,\quad \omega_2=f_2\rd t\wedge \eta_2+f_3f_1 \eta_3\wedge \eta_1,
\quad \omega_3=f_3\rd t\wedge \eta_3+f_1f_2 \eta_1\wedge \eta_2,
\end{equation}
for a triple of $t$-valued functions $f=(f_1,f_2,f_3)$.  This triple
automatically satisfies the orthogonality conditions
\eqref{omegas_orthogonal} provided that $f_1f_2f_3\neq 0$, and so
will define a hyperk\"ahler structure if $\rd \omega_i=0$. This is
equivalent to the following system of ODEs:
\begin{equation}\label{ODEs}
\frac{\rd f_1}{\rd t}=\frac{f_2^2+f_3^2-f_1^2}{f_2f_3},\quad  
\frac{\rd f_2}{\rd t}=\frac{f_3^2+f_1^2-f_2^2}{f_3f_1},\quad
\frac{\rd f_3}{\rd t}=\frac{f_1^2+f_2^2-f_3^2}{f_1f_2}.
\end{equation}
There are then three possibilities.
\begin{itemize}
\item
When $f_1=f_2=f_3$, one quickly obtains the standard flat metric on $\RR^4$ and the standard triple where $f_1=f_2=f_3=t=r$, the radial distance from the origin.   The closed framings of the 2-forms on the $S^3$ orbits of the action are in this case simply 
\[
r^2(\eta_2\wedge\eta_3,\eta_3\wedge\eta_1,\eta_1\wedge\eta_2).
\]  
\item When $f_1\neq f_2=f_3=r$, then one finds that $f_1=r(1-c^4/r^4)^{\frac{1}{2}}$ for a constant $c>0$ where $r\geq c$, which gives the Eguchi--Hanson metric on $T^*S^2$, given for $r>c$ by:
\[
\left(1-\frac{c^4}{r^4}\right)^{-1}\rd r^2+r^2\left(1-\frac{c^4}{r^4}\right)\eta_1^2+r^2\eta_2^2+r^2\eta_3^2.
\]
  The orbits where $r>c$ is constant are
  $\RR\bP^3$s, whereas the exceptional orbit where $r=c$ gives the
  $S^2$ ``bolt'' which is the zero section.  The closed framings of
  $\Lambda^2T^*\RR\bP^3$ are: 
\[
(r^2\eta_2\wedge\eta_3, r^2(1-c^4/r^4)^{\frac{1}{2}}\eta_3\wedge\eta_1,r^2(1-c^4/r^4)^{\frac{1}{2}}\eta_1\wedge\eta_2).
\]  
As $r\to\infty$ these framings approach the standard closed framing on $\RR\bP^3$.  The induced metric on each $\RR\bP^3$ is a Berger metric $r^2(1-c^4/r^4)\eta_1^2+r^2\eta_2^2+r^2\eta_3^2$, where the relative ``squashing'' of the circle corresponding to $\eta_1$ can take any value in $(0,1)$.  Taking the same closed framings on $S^3$ will not lead to a complete invariant hyperk\"ahler metric, but instead to a double cover of the Eguchi--Hanson space. 
\item When all of the $f_i$ are distinct, one does not obtain a complete metric.
\end{itemize}

If one now wants to study invariant hyperk\"ahler metrics where the action rotates the frame one views the standard left-invariant coframe on $\SU(2)$ as a 1-form taking values in the imaginary quaternions (rather than $\RR^3$).  If we also identify points $q\in \SU(2)\cong S^3$ with unit quaternions, we may define a triple $\hat{\omega}$ of 2-forms by
$\hat{\omega}|_q=q\omega q^{-1}$, where $\omega$ is as in \eqref{e1.13.8.15} (now
viewed as taking values in the imaginary quaternions).  This time, in
place of \eqref{ODEs}, we obtain
\begin{equation}\label{ODEs2}
\frac{\rd f_1}{\rd t}=\frac{f_2^2+f_3^2-f_1^2-2f_2f_3}{f_2f_3},\!\quad\!  
\frac{\rd f_2}{\rd t}=\frac{f_3^2+f_1^2-f_2^2-2f_3f_1}{f_3f_1},\!\quad\!
\frac{\rd f_3}{\rd t}=\frac{f_1^2+f_2^2-f_3^2-2f_1f_2}{f_1f_2}.
\end{equation}
Again, there are three possibilities.
\begin{itemize}
\item
When $f_1=f_2=f_3$, one unsurprisingly again obtains the flat hyperk\"ahler metric on $\RR^4$ since $f_1=f_2=f_3=-t$.
\item
When $f_1\neq f_2=f_3$, one can solve \eqref{ODEs2} with $f_1=2m(r-m)^{1/2}(r+m)^{-1/2}$ and $f_2=f_3=(r^2-m^2)^{1/2}$ 
for a constant $m>0$  where $r\geq m$.  This leads again to a metric on $\RR^4$ which now has cubic volume growth at infinity, known as the Taub--NUT metric (with ``mass'' $m$):
\[
\frac{1}{4}\frac{r+m}{r-m}\rd r^2+4m^2\frac{r-m}{r+m}\eta_1^2+(r^2-m^2)\eta_2^2+(r^2-m^2)\eta_3^2.
\]
So we see on the hyperspheres where $r$ is constant the induced metric, as for the Eguchi--Hanson metric, is a Berger metric where the relative ``squashing'' of the circle factor on $S^3$ corresponding to $\eta_1$ can again take any value in $(0,1)$.  This shows that although the metrics on $S^3$ here and in the double cover of the Eguchi--Hanson metric are the same, the closed framings of the bundle of 2-forms are different so that one finds different hyperk\"ahler triples (and hence metrics) extending them, as we must have by Theorem \ref{bryant.thm}.  One can describe the closed framings of $\Lambda^2T^*S^3$ as $q\gamma q^{-1}$ where $\gamma$ is the triple
\[
\big((r^2-m^2)\eta_2\wedge\eta_3, 2m(r-m)\eta_3\wedge\eta_1, 2m(r-m)\eta_1\wedge\eta_2\big).
\]
\item When the $f_i$ are all distinct, one can solve explicitly \eqref{ODEs2} using elliptic functions and obtain the Atiyah--Hitchin metric, defined on $S^4\setminus \RR\bP^2$, which arises in the study of moduli spaces of monopoles on $\RR^3$.  The metric near the Veronese $\RR\bP^2$ at infinity is asymptotic to the Taub--NUT metric with mass $m<0$.   Here the orbits are $\SO(3)/(\ZZ_2\times\ZZ_2)$, except for an exceptional $\RR\bP^2$ orbit, and one can write the induced closed framings of the 2-forms in terms of elliptic functions, and also observe that the induced metrics are no longer Berger metrics.
  (One may also consider the double cover of the Atiyah--Hitchin
  metric on $\bCP^2\setminus S^2$, that has $\SU(2)/\ZZ_4$ as the orbits of the
  action except for a special $S^2$ orbit, and which now can be deformed in a
  1-parameter family of gravitational instantons that are not
  $\SU(2)$-invariant \cite{Dancer}.)
\end{itemize}

We focus on the simplest example of a hyperk\"ahler 4-manifold with boundary arising from this analysis, namely the unit 4-ball with the flat metric.  
We let $\eta$ be the standard left-invariant coframe on $S^3$, let $\omega$ be the standard hyperk\"ahler triple on $B^4$ and let $\gamma=\omega|{S^3}$.

  As we have seen, a key point is to study the closed anti-self-dual 2-forms.  We see that in this case they are simple to describe explicitly using negative frequency data on the boundary.

\begin{lem}\label{B4.ASD}  Let $E_k=\{\alpha\in \Omega^2(S^3):\rd *\alpha=-k\alpha\}$ for $k\in\bN\setminus\{1\}$.
Closed anti-self-dual 2-forms on $B^4$ are given by $\sum_{k=2}^{\infty} \rd(r^{k}*\alpha_k)$ where $\alpha_k\in E_k$. 
Hence, $\cZ^2_-(B^4)$ is isomorphic to
\[
\cZ^2_-(S^3)=\{\alpha\in \Omega^2(S^3):\rd\alpha=0,\,\alpha
\in\oplus_{k=2}^{\infty} E_k\}.
\]
\end{lem}

\begin{proof}
The eigenvalues of $\rd *$ on closed 2-forms on $S^3$ are well-known to be $k\in\ZZ\setminus\{0,\pm 1\}$ with multiplicity $k^2-1$.  The result then follows from the one-to-one correspondence between closed anti-self-dual 2-forms on $B^4$ and 
eigenforms for $\rd *$ on $S^3$.
\end{proof}

This lemma together with our main results allow us to explicitly describe the moduli space of hyperk\"ahler triples on 
$B^4$ in terms of boundary values on $S^3$ as follows.  Notice that, by Theorem \ref{smooth_moduli_space}, the true 
space of hyperk\"ahler deformations of the flat metric on $B^4$, working up to the action of diffeomorphisms which can move the boundary $S^3$,  is 
described by the quotient $\cT/L(\cW)$ where $\cT=[\cZ^2_-(B^4)\otimes\RR^3]+L(\cW)$. 

\begin{prop}
On $B^4$, $\cT/L(\cW)\cong\{\alpha\in \cZ^2_-(S^3)\otimes\RR^3: (\alpha_i,\gamma_j)\in C^{\infty}(S^3,S^2_0\RR^3)\}.$ 
\end{prop}

\begin{proof}
Elements $Lv\in L(\cW)_-=L(\cW)\cap [\cZ_-^2(B^4)\otimes\RR^3]$ satisfy $L_+v=0$, which is equivalent to the 
Dirac equation $Dv=0$ by Proposition \ref{linearisation.vector.field}, and thus are determined by the boundary values of $v$.  Moreover, we know that $Lv$ is given as a sum of forms which are homogeneous in $r$ by Lemma \ref{B4.ASD}.  We thus restrict to the case where $v=r^kf\partial_r+r^{k-1}w$ where $f$ is a function on $S^3$,  $w$ is a vector field on $S^3$ and $k\in\bN$.  We calculate from the equation $L_+v=0$ that we have (viewing $w$ as a 1-form)
\[
\rd^*w=(k+1)f\quad\text{and}\quad \rd f-*\rd w=(k+3)w.
\]
We deduce that, recalling that $\eta$ is the standard coframe on $S^3$, 
\[
Lv|{S^3}=L_-v|{S^3}=*\rd(i_w\eta)-(k+1)f\gamma+(k+3)w\wedge \eta.
\]
Hence, $\alpha\in\cZ^2_-(Y)$ is $L^2$-orthogonal  to $Lv|{S^3}$ 
if and only if 
\[
\langle\alpha,*\rd(i_w\eta)-(k+1)f\gamma+(k+3)w\wedge\eta\rangle
=-(k+1)f\langle\alpha,\gamma\rangle+(k+3)\langle\alpha,w\wedge\eta\rangle=0,
\]
since $\rd\alpha=0$.  
Hence, by imposing this condition for all $Lv\in L(\cW)_-$, which amounts to varying $f$ and $w$ (and hence $k$) so that $Lv|{S^3}\in\cZ^2_-(S^3)$, we deduce that we must have, in summation convention:
$\alpha_i\wedge\eta_i=0$ and $\epsilon_{ijk}\alpha_j\wedge\eta_k=0$.
This is equivalent to the vanishing of the 
trace and skew-parts of the matrix $(\alpha_i,\gamma_j)$ of 
inner products.
\end{proof}

It is natural to ask what happens when one takes positive frequency data on
 $S^3$ instead.  One knows that this cannot fill in to a hyperk\"ahler triple on
  $B^4$, but in general one cannot say more than that.  However, in a special 
  case we can explicitly demonstrate that we can take arbitrarily small positive 
  frequency data which has no hyperk\"ahler filling, by 
relating the deformation of the boundary data to the Eguchi--Hanson metric.

\begin{prop}\label{non.exist.prop}
  For $c\in (0,1)$ let 
\[
\hat{\gamma}=(\eta_2\wedge\eta_3, (1-c^4)^{\frac{1}{2}}\eta_3\wedge\eta_1,(1-c^4)^{\frac{1}{2}}\eta_1\wedge\eta_2).
\]
Then $\hat{\gamma}-\gamma$ has positive frequency with respect to $\rd*$ on $S^3$ and there does not exist a 
hyperk\"ahler triple $\hat{\omega}$ on $B^4$ such that $\hat{\omega}|{S^3}=\hat{\gamma}$.
\end{prop}

\begin{proof}
We see that $\hat{\gamma}-\gamma$ consists of constant multiples of $\eta_i\wedge\eta_j$ which have eigenvalue 2 
with respect to $\rd *$ and so are positive frequency.  We know that $\hat{\gamma}$ has a unique hyperk\"ahler extension
$\hat{\omega}$ given by the ODEs \eqref{ODEs} derived above, which lead to the Eguchi--Hanson triple
\[
\hat{\omega}=(r\rd r\wedge\eta_1+r^2\eta_2\wedge\eta_3,rf^{-1}\rd r\wedge\eta_2+ r^2f\eta_3\wedge\eta_1,rf^{-1}\rd r\wedge\eta_3+r^2f\eta_1\wedge\eta_2)
\]
where $f(r)=(1-c^4/r^4)^{\frac{1}{2}}$ for $r>c$.  The issue is whether this can be extended smoothly to $r=c$ to give a 
hyperk\"ahler metric on $B^4$, but this is not possible by the classification of $\SU(2)$-invariant hyperk\"ahler 4-manifolds.    
\end{proof}

As we saw above, there are two $\SU(2)$-invariant hyperk\"ahler metrics on $B^4$: the flat metric and the Taub--NUT metric.  We have an induced closed framing of the 2-forms on $S^3$ in Taub--NUT when $r=m+\frac{1}{2m}$ given by $q\hat{\gamma}q^{-1}$ where
\[
\hat{\gamma}=\big((1+\frac{1}{4m^2})\eta_2\wedge\eta_3,\eta_3\wedge\eta_1,\eta_1\wedge\eta_2).
\]
Hence, if we consider the second standard framing of the 2-forms on
$S^3$ given by $q\gamma q^{-1}$, then $\hat{\gamma}-\gamma=
(4m^2)^{-1}(\eta_2\wedge\eta_3,0,0)$, which can clearly be made
arbitrarily small by making the mass $m$ sufficiently large.  Notice
that in Proposition~\ref{non.exist.prop} this difference was seen to be positive
frequency with respect to $\gamma$.  However, we observe that in the
analysis of the $\SU(2)$-invariant hyperk\"ahler 4-manifolds above
that  
the induced orientation on $S^3$ is such that $-\eta_1\wedge\eta_2\wedge\eta_3>0$ (i.e.~reversed). Hence, $q(\hat{\gamma}-\gamma)q^{-1}$ is now negative frequency as we would expect.

We summarise this discussion in a final proposition.

\begin{prop}
Let $S^3$ be endowed with the closed framing of the $2$-forms given at $q\in S^3$ by $q\gamma q^{-1}$.  For any $m>0$ there exist closed framings $q\hat{\gamma}q^{-1}$ of the $2$-forms on $S^3$ such that $q(\hat{\gamma}-\gamma)q^{-1}$ is negative frequency with respect to $\rd*$ on $S^3$ and the hyperk\"ahler filling of 
$q\hat{\gamma}q^{-1}$ to $B^4$ is given by Taub--NUT with mass $m$.  
\end{prop}


\thebibliography{99}

\bibitem{AHS} M.~F.~Atiyah, N.~J.~Hitchin and I.~M.~Singer, {\it Self-duality in four-dimensional Riemannian geometry},
Proc.~Roy.~Soc.~London Ser.~A  {\bf 362}  (1978), 425--461. 

\bibitem{MR934202} M.~F.~Atiyah and N.~J.~Hitchin, 
{\it The geometry and dynamics of magnetic monopoles},
M.~B.~Porter Lectures,
  Princeton University Press, Princeton, NJ, 1988. 



\bibitem{BaerBallmann} C.~B\"ar and W.~Ballmann, {\it Guide to boundary value problems for Dirac-type operators}, arXiv:1307.3021.

\bibitem{BartnikChrusciel} R.~A.~Bartnik and P.~T.~Chrusciel, {\it Boundary value problems for Dirac-type equations}, J. Reine Angew. Math. {\bf 579} (2005), 13--73.

\bibitem{Biquard} O.~Biquard, {\it M\'etriques autoduales sur la boule}, Invent.~ Math.~{\bf 148} (2002), 545--607.

\bibitem{Bryant} R.~L.~Bryant, {\it Nonembedding and nonextension results in special holonomy}, in The many facets of geometry, pp.~346--367, Oxford University Press, Oxford, 2010.

\bibitem{Cartan} \'E.~Cartan, {\it La g\'eometrie des espaces de {R}iemann}, M\'emorial des Sciences Mathematiques, Fasc.~IX (1925).

\bibitem{Dancer} A.~S.~Dancer, {\it A family of hyperk\"ahler manifolds}, 
Q.~J.~Math.~{\bf 45} (1994), 463--478.

\bibitem{Donaldson} S. K. Donaldson, {\it Two-forms on four-manifolds and elliptic equations}, in Inspired by S.~S.~Chern, volume 11 of Nankai Tracts Math., pp.~153--172, World Sci. Publ., Hackensack, NJ, 2006.

\bibitem{Ebin} D.~Ebin, {\it The manifold of {R}iemannian metrics}, Proc.~Symp.~Pure Math.~{\bf 15} (1970), 11--40.

\bibitem{EbinMarsden}  D.~Ebin and J.~Marsden, {\it Groups of
    diffeomorphisms and the motion of an incompressible fluid},
  Ann.~of Math.~{\bf 92} (1970), 102--163.

\bibitem{GibbonsPope} G.~Gibbons and  C.~Pope, {\it The positive action conjecture and
asymptotically Euclidean metrics in quantum gravity}, Commun.
Math. Phys. {\bf 66} (1979), 267-290.

\bibitem{Gromov} M.~Gromov, {\it Partial Differential Relations}, Springer, 1986.

\bibitem{HMZ} O.~Hijazi, S.~Montiel and X.~Zhang, {\it Dirac operator on embedded hypersurfaces}, Math.~Res.~Lett.~{\bf 8} (2001), 195--208.

\bibitem{Hitchin} N.~Hitchin, {\it Hyper-K\"ahler manifolds},
S\'eminaire Bourbaki, Vol.~1991/92,
 Ast\'erisque  {\bf 206}  (1992), Exp.~No.~748, 3, 137--166.

\bibitem{Claude_thesis} C.~LeBrun, {\it Spaces of complex geodesics and related structures}, DPhil, University of Oxford, 1980. Available at 
\href{http://ora.ox.ac.uk/objects/uuid:e29dd99c-0437-4956-8280-89dda76fa3f8}{\texttt{http://ora.ox.ac.uk/objects/uuid:e29dd99c-0437-4956-8280-89dda76fa3f8}}.

\bibitem{LeB91} C.~LeBrun, {\it On complete quaternionic-K\"ahler manifolds}, 
Duke Math.~J.~{\bf 63} (1991), 723--743. 

\bibitem{Schwarz} G.~Schwarz, {\it Hodge decomposition -- A method for solving boundary value problems},
Springer-Verlag, Berlin, 1995.

\bibitem{Taylor} M.~E.~Taylor, {\it Partial differential equations}, Vol.~1, Applied Math.~Sciences {\bf 115}, Springer-Verlag, New York, 1996.

\bibitem{Tromba} A.~J.~Tromba, {\it Teichm\"uller theory in Riemannian geometry}, Lectures in Mathematics ETH Z\"urich, Birkh\"auser Verlag, Basel, 1992.
\end{document}